\documentclass[12pt]{amsart}

\title[Conformal dimension and random groups]{Conformal dimension and random groups}
\author{John M. Mackay}
\address{Department of Mathematics \\
 University of Illinois at Urbana-Champaign \\ Urbana IL.}
\curraddr{Mathematical Institute, 24-29 St Giles', Oxford OX1 3LB, UK.}
\email{jmmackay@illinois.edu}
\date{\today}

\usepackage{amsmath,amsthm,amsfonts,graphicx,amssymb}
\usepackage{hyperref}

\numberwithin{equation}{section}
\newtheorem{theorem}[equation]{Theorem}
\newtheorem{proposition}[equation]{Proposition}

\newtheorem{lemma}[equation]{Lemma}

\newtheorem{definition}[equation]{Definition}

\newtheoremstyle{citing}
  {3pt}
  {3pt}
  {\itshape}
  {}
  {\bfseries}
  {}
  {.5em}
  {\thmnote{#3}}

\theoremstyle{citing}
\newtheorem*{varthm}{}

\makeatletter

\newenvironment{qedequation*}{%
  \pushQED{\qed}%
  \mathdisplay@push
  \st@rredtrue \global\@eqnswfalse
  \mathdisplay{equation*}%
}{%
  \endmathdisplay{equation*}%
  \mathdisplay@pop
  \ignorespacesafterend
  \popQED\@endpefalse
}
\makeatother

\DeclareMathOperator{\AC}{AC}
\DeclareMathOperator{\St}{St}

\newcommand{\gam}{\gamma}
\newcommand{\eps}{\epsilon}

\newcommand{\sig}{\sigma}
\newcommand{\bdry}{\partial_\infty}

\newcommand{\cC}{\mathcal{C}}
\newcommand{\cD}{\mathcal{D}}

\newcommand{\Cdim}{\mathcal{C}\mathrm{dim}}

\newcommand{\dimP}{\dim_{\mathcal{P}}}
\newcommand{\cP}{\mathcal{P}}

\newcommand{\ra}{\rightarrow}

\newcommand{\N}{\mathbb{N}}
\newcommand{\Z}{\mathbb{Z}}
\newcommand{\ba}{\mathbf{a}}

\def\XXint#1#2#3{{\setbox0=\hbox{$#1{#2#3}{\int}$}
\vcenter{\hbox{$#2#3$}}\kern-.5\wd0}}

\numberwithin{equation}{section}

\begin{document}

\begin{abstract}
We give a lower and an upper bound for the conformal dimension 
of the boundaries of certain small cancellation groups.

We apply these bounds to the few relator and density models for
random groups.
This gives generic bounds of the following form, where $l$ is the relator length, going to infinity.

(a) $1 + 1/C < \Cdim(\bdry G) < C l / \log(l)$, for the few relator model, and

(b) $1 + l / (C\log(l)) < \Cdim(\bdry G) < C l$, for the density model,
at densities $d < 1/16$.

In particular, for the density model at densities $d < 1/16$, as the relator
length $l$ goes to infinity, the random groups will pass through infinitely many
different quasi-isometry classes.
\end{abstract}

\maketitle

\section{Introduction}\label{sec-intro}

\subsection{Overview}
In the study of random groups, one considers typical properties 
of finitely presented groups.
There are several ways to make this idea precise.
We will work in two of the most common models
for a random group: the few relator model and the density model,
both due to Gromov.  Our goal is to study the large scale geometry of such groups.

In each of these models, a typical group is (Gromov) hyperbolic~\cite{Gro-87-hyp-groups, Gro-91-asymp-inv}.
To any hyperbolic group $G$ we can associate a boundary at infinity $\bdry G$,
which is a metric space where the metric is canonically defined
up to a quasi-symmetric homeomorphism.
For example, we can choose any visual metric on the boundary.
The boundary captures the quasi-isometry type of the group:
two finitely presented hyperbolic groups are quasi-isometric if and only if
their boundaries are quasi-symmetric~\cite{Pau-96-qm-qi}.

In both models, the boundary of a random group is
homeomorphic to the Menger curve (also called the Menger sponge),
but this is not enough to determine the quasi-symmetric type 
of the boundary, and hence the large scale geometry of the group.
Indeed, Bourdon found an infinite family of hyperbolic groups
whose boundaries are all homeomorphic to the Menger curve,
but which are pairwise non-quasi-isometric~\cite{Bou-97-GAFA-exact-cdim}.

The invariant which Bourdon used, and which we will use in this paper,
is due to Pansu.
The \emph{conformal dimension} of a metric space $X$ is 
the infimal Hausdorff dimension of all metric spaces 
quasi-sym\-met\-ri\-cal\-ly equivalent to $X$, and is denoted by 
$\Cdim(X)$~\cite{Pan-89-cdim, Mac-Tys-cdimexpo}.
Recall that a homeomorphism is quasi-symmetric if there is 
uniform control on how it distorts 
annuli~\cite[(1.2)]{TV-80-qs}. (See also~\cite[Chapter 10]{Hei-01-lect-analysis}.)

Conformal dimension is clearly a quasi-symmetric invariant of a metric space, and 
consequently the conformal dimension of $\bdry G$
is canonically defined and depends only on the quasi-isometry type of the group.
(For more discussion of this invariant, see \cite{Kle-06-ICM,Mac-Tys-cdimexpo}.)
Bourdon's family of groups have boundaries with conformal dimension
attaining values in a dense subset of $(1,\infty)$~\cite[Th\'eor\`eme 1.1]{Bou-97-GAFA-exact-cdim},
and so they lie in infinitely many different quasi-isometry classes.

In this paper we give the first significant progress towards calculating the
conformal dimension of the boundary of a random group.
(A question raised by Gromov and Pansu; see the following subsection.)
In particular, we show that in the density model with $d<1/16$, as the lengths of the relators
tend to infinity, the conformal dimension of the boundary also tends to infinity,
passing through infinitely many different quasi-isometry types.

\subsection{Statement of results}

The simplest model of a random group is given by the
few relator model.  
Throughout this paper we fix a finite generating set $S$, $|S| = m \geq 2$. 

\begin{definition}[Few relator model]\label{def-gen-few-rel}
	Fix a finite number of relators $n\geq 1$.
	Consider all cyclically reduced words of length at most $l$ in $\langle S \rangle$.
	Consider all presentations $\langle S \mid r_1, \ldots, r_n \rangle$
	where $r_1, \ldots, r_n$ are chosen from this set of words 
	uniformly and independently at random.

	A property $\cP$ is \emph{generic} in the few relator model (for fixed $n$),
	if the proportion of all such presentations at length $l$ which satisfy 
	$\cP$ goes to $1$ as $l \ra \infty$.  
	In this case, we say that \emph{a random (few relator) group has property $\cP$}.
\end{definition}

This model was introduced by Gromov~\cite{Gro-87-hyp-groups},
who observed that a random few relator group
will satisfy the $C'(1/6)$ small cancellation condition, and so be hyperbolic.
The algebraic properties of these groups, such as freeness of subgroups and
isomorphism type, have been studied by Arzhantseva, Ol'shanskii,
Kapovich, Schupp, and others \cite{Arz-Ol-generic-subgroups-free,KS-05-rand-isom-one-rel,KSS-06-rand-isom-rigid}.  
For more discussion, see~\cite[I.3.c]{Oll-05-rand-grp-survey}.

The geometry of such groups was considered by Champetier~\cite{Cha-95-rand-grps}.
He used small cancellation techniques to show that generic 
few relator groups have boundaries homeomorphic to the Menger curve
(see Theorem~\ref{thm-champ-twelth-menger}).

The few relator model can be viewed as the ``density $0$'' case of a more general
model, where the number of relators grows as $l \ra \infty$.

\begin{definition}[Density model {\cite[Chapter 9]{Gro-91-asymp-inv}}]\label{def-gen-density}
	Fix a parameter $d \in (0,1)$, called the \emph{density}.
	Consider all cyclically reduced words of length $l$ in $\langle S \rangle$.
	Consider all presentations which choose as relators
	$(2m-1)^{dl}$ of these words uniformly and independently
	at random.

	A property $\cP$ holds generically in the density model (at fixed density $d$),
	if the proportion of all such presentations at length $l$ which satisfy 
	$\cP$ goes to $1$ as $l \ra \infty$.
\end{definition}

Gromov showed that the density model has a phase transition:
for densities $d< 1/2$, a random group will be one ended and hyperbolic,
but for densities $d > 1/2$, a random group will be trivial or $\Z/2\Z$
({\cite[Section 9.B]{Gro-91-asymp-inv}}, {\cite[Theorem 11]{Oll-05-rand-grp-survey}}).
%

The boundary of a random group at density $d< 1/2$ 
is homeomorphic to the Menger curve.
At densities $d<1/24$, this follows from Champetier's Theorem~\ref{thm-champ-twelth-menger}.
A proof that applies to all densities $0 < d < 1/2$ is given in \cite{DGP-10-density-menger}.

Since we know that random groups in both the few relator and density model
are hyperbolic, it makes sense to ask for estimates of the conformal dimension
of their boundaries.

For any hyperbolic group with boundary homeomorphic to the Men\-ger curve,
the conformal dimension of the boundary will be strictly greater than one~\cite{Mac-10-confdim},
and finite~\cite{Coo-93-meas-bdry}.
In this paper we give explicit non-trivial bounds for the conformal dimension of a random group.
\begin{theorem}\label{thm-main-few-rel}
	There exists $C>1$ so that,
	for fixed $m \geq 2$, $n \geq 1$, the
	conformal dimension of a random few relator group satisfies
	\begin{equation*}
		1 + \frac{1}{C} \leq \Cdim(\bdry G) \leq	C \log(2m-1) \cdot \frac{l}{\log(l)}.
	\end{equation*}
\end{theorem}
Note that the lower bound is independent of $l$, and the upper bound is sub-linear.
The conclusion of this theorem involves $n$ implicitly: let
$P(m,n,l)$ be the proportion of all groups with $m$ generators and
$n$ cyclically reduced relators of word length at most $l$ which
satisfy the above estimate.  Then for fixed $m, n$ we have
$P(m,n,l) \ra 1$ as $l \ra\infty$, however the rate of convergence depends on $n$.

\begin{theorem}\label{thm-main-density}
	There exists $C>1$ so that, for fixed $m \geq 2$, $0 < d < 1/16$, the
	conformal dimension of a random group at density $d$ satisfies
	\begin{equation*}
		1 + \frac{d \log(2m)}{C} \cdot \frac{l}{\log(l)} \leq \Cdim(\bdry G) \leq \frac{C\log(2m-1)}{|\log(d)|} \cdot l.
	\end{equation*}
	
	In particular, as $l \ra \infty$, generic groups pass through 
	infinitely many different quasi-isometry classes.
\end{theorem}

Gromov \cite[9.B, p.276, (g)]{Gro-91-asymp-inv} and 
Pansu \cite[IV.b., p.70]{Oll-05-rand-grp-survey}, had asked
whether the conformal dimension of a random group in the density model 
can be used to detect the particular density $d$.
Theorem~\ref{thm-main-density}, gives progress towards solving this problem.

Roughly speaking, the few relator model corresponds to taking $d = 1/l$.
We see this reflected in factors of $1/\log(l)$ and $1/|\log(d)|$ in
the upper bounds of Theorems~\ref{thm-main-few-rel} and \ref{thm-main-density}.

There are several natural questions that remain.
For example, does the conformal dimension of a random few relator group go to
infinity as the relator length goes to infinity?
What happens at densities $d \geq 1/16$? (See the discussion below.)
Can one find a function $f(d,l)$ so that the conformal dimension of
a random group at density $d$ satisfies
$f(d,l) \lesssim \Cdim(\bdry G) \lesssim f(d,l)$?
(We write $x \lesssim y$ if $x \leq C y$, for some suitable constant $C$.)


For more background on random groups we refer the reader 
to \cite{Ghys-rand-grp-survey} and \cite{Oll-05-rand-grp-survey},
and on conformal dimension to \cite{Mac-Tys-cdimexpo}.

\subsection{Outline of proof}

The random groups that we consider are all $C'(1/6)$ small cancellation groups.
Recall that a group presentation $\langle S | R \rangle$
is $C'(\lambda)$ if every word $u$ which appears in two distinct ways in
(cyclic conjugates of) relators $r_1, r_2 \in R$, or their inverses, satisfies
$|u| < \min\{|r_1|,|r_2|\}$.

In the few relator model this is straightforward to prove;
a more refined estimate is found in Proposition~\ref{prop-lambda-gen-upper},
where we show that a generic few relator group will be
$C'(\lambda)$ with $\lambda \lesssim \log(l)/{l}$.
In the density model, we have the following result.
\begin{proposition}[{\cite[Section 9.B]{Gro-91-asymp-inv}}]\label{prop-dens-sc} 
	For $d>0$, and $\lambda > 2d$,
	a random group at density $d$ has
	the $C'(\lambda)$ metric small cancellation condition.
	
	For $\lambda >0$ and $2d > \lambda$,
	a random group at density $d$ does not have 
	the $C'(\lambda)$ metric small cancellation condition.
\end{proposition}
In particular, at densities $d<1/12$, a random group has a
$C'(1/6)$ small cancellation presentation.

Specifying a finite generating set $S$ for $G$ allows one to define the
Cayley graph $\Gamma=\Gamma(G,S)$.
The Cayley graph of a $C'(1/6)$ group is $\delta$-hyperbolic, with
$\delta$ equal to twice the maximum relator word length (Lemma~\ref{lem-thin-triangles}).

As a hyperbolic metric space, for any sufficiently small visual parameter $\eps>0$,
$\bdry \Gamma$ carries a visual metric comparable to $e^{-\eps(\cdot,\cdot)}$,
where $(\cdot,\cdot)$ denotes the Gromov product.
A simple upper bound on the conformal dimension of $\bdry G$ is given 
by the Hausdorff dimension of this metric space, which 
equals $h(G)/\eps$~\cite[Corollary 7.6]{Coo-93-meas-bdry}.
Here $h(G)$ is the volume entropy of the group (with respect to $S$).
In an $m$-generator group, we always have $h(G) \leq \log(2 m -1)$.
(This is essentially sharp for a random group \cite{Shu-99-rand-growth,Oll-06-rand-growth}.)
Thus,
\begin{equation}\label{eq-cdim-eps-entropy}
	\Cdim(\bdry G) \leq \frac{1}{\eps} h(G) \leq \frac{1}{\eps} \log(2m-1).
\end{equation}

To give a good upper bound for the conformal dimension, then, we would like to choose
$\eps$ as large as possible.  The standard estimate for an admissible
$\eps$ is $\eps \leq \log(2) / (4 \delta)$ \cite[III.H.3.21]{BH-99-Metric-spaces}.
In the few relator model, or the density model with $d < 1/12$, we can take $\delta = 2l$.
Ollivier shows that in the density model with $d<1/2$ we can take $\delta \leq 4l/(1-2d)$
\cite[Corollary 3]{Oll-07-sc-rand-group}.

Consequently, in the few relator model we have the generic estimate
$\Cdim(\bdry G) \leq (8 \log(2m-1)/\log(2)) \cdot l$.
The following result is also immediate, and worth noting.
\begin{proposition}
	A random group at density $d<1/2$ will satisfy
	\begin{equation*}
		\Cdim(\bdry G) \leq \frac{16 \log(2m-1)}{(\log 2)(1-2d)} \cdot l.
	\end{equation*}
\end{proposition}
(Theorem~\ref{thm-main-density} gives a sharper upper bound for small $d$.)

For $C'(1/6)$ groups, we cannot find a significantly better estimate for $\delta$, 
since the relators of size $l$ give bigons
with sides separated by a distance of order $l$.
However, work of Bonk and Foertsch~\cite{BF-06-ACu-kappa} lets us find a better estimate for $\eps$ using the concept
of ``asymptotic upper curvature'' (Section~\ref{sec-ac-upper-cdim}).
\begin{varthm}[Theorem \ref{thm-sc-cdim-upper}.]
	If $G = \langle S | R \rangle$ is a $C'(\lambda)$ presentation of a group, with
	$\lambda \leq 1/6$, and $|r| \leq M$ for all $r \in R$, then
	\[
		\Cdim(\partial_\infty(G))
			\leq \frac{M}{2 \log \lfloor\frac{1}{\lambda}-4 \rfloor} \log(2m-1).
	\]
\end{varthm}
Combining this theorem with Propositions~\ref{prop-lambda-gen-upper}~and~\ref{prop-dens-sc},
we obtain the upper bounds in Theorems~\ref{thm-main-few-rel}~and~\ref{thm-main-density}.
 
It is more difficult to obtain lower bounds for the conformal dimension.
A key inspiration for our work is the following result of Champetier.
\begin{theorem}[{\cite[Theorem 4.18]{Cha-95-rand-grps}}]\label{thm-champ-twelth-menger}
	Suppose $G = \langle S | R \rangle$ is a $C'(1/12)$ presentation, 
	with $|S| \geq 2$ and $|R| \geq 1$.
	Suppose further that 
	every reduced word $u \in \langle S \rangle$ of length $12$
	appears at least once in some cyclic conjugate of some $r^{\pm 1}, r \in R$.
	Then $\bdry G$ is homeomorphic to the Menger curve.
\end{theorem}

Random groups certainly contain every word of length $12$ as a subword of some relator.
In fact, generic few relator presentations contain every word of length
$C \log(l)$ as a subword of some relator (Proposition~\ref{prop-few-rel-wordlength}),
while generic presentations at density $d$ contain every word of length $C l$, for $C<d$, as
a subword of some relator (Proposition~\ref{prop-dens-wordlength}).

Champetier builds a cone in the Cayley complex of a $C'(1/12)$
group that gives an arc in its boundary.  
We strengthen his techniques slightly to $C'(1/8 - \delta)$ groups,
and produce instead a sub-complex
quasi-isometric to one of Gromov's ``round trees'' \cite{Gro-91-asymp-inv, Bou-95-cdim-hyp-build}
This gives a Cantor
set of curves in the boundary, to which we apply a lemma of Pansu and Bourdon.
We find the following lower bound for the conformal dimension of a group in terms of
simple algebraic properties of its presentation.
\begin{varthm}[Theorem \ref{thm-sc-cdim-lower}.]
	Suppose $G = \langle S | R \rangle$ is a $C'(1/8 -\delta)$ presentation, 
	with $|S| = m \geq 2$ and $|R| \geq 1$, where $\delta \in (0, 1/8)$
	and $|r| \in [3/\delta, M]$ for all $r \in R$.
	Suppose further that for some $M^* \geq 12$, 
	every reduced word $u \in \langle S \rangle$ of length $M^*$
	appears at least once in some cyclic conjugate of some relator $r^{\pm 1}, r \in R$.
	Then for some universal constant $C>0$, we have
	\[
		\Cdim(\bdry G) \geq 1 + C\log(2m) \cdot \frac{M^*}{\log(M)}.
	\]
	(If we have a $C'(1/11)$ presentation, the lower bound on
	the lengths of relators holds automatically.)
\end{varthm}

This theorem combines with Proposition~\ref{prop-few-rel-wordlength} and \ref{prop-dens-wordlength}
to complete the proof of Theorems~\ref{thm-main-few-rel} and \ref{thm-main-density}.

Random groups in the density model have better small cancellation
properties than their optimal $C'(\lambda)$ condition would lead you to expect
(see Ollivier~\cite{Oll-07-sc-rand-group} and 
Ollivier and Wise~\cite{Oll-Wis-11-rand-grp-T}).
Using results from these papers, the author has extended Theorems~\ref{thm-main-density} and \ref{thm-sc-cdim-lower}
to densities $d<1/13$.
It is reasonable to expect that similar techniques to ours may be used to find a good lower bound for the
conformal dimension of a random group at densities up to, say, $d<1/6$,
however entirely different techniques would be needed above $d>1/4$, as at these densities random groups
have no good small cancellation properties at all.

\subsection{Outline of paper}

In Section~\ref{sec-rnd-grp-sc} we consider random groups in both models
and their small cancellation properties.
Standard results about the geometry of $C'(1/6)$ groups, including hyperbolicity,
are given in Section~\ref{sec-sc-cayley}.

Asymptotic upper curvature bounds are used in Section~\ref{sec-ac-upper-cdim}
to give a generic upper bound for conformal dimension.
A round tree sub-complex is built in Section~\ref{sec-round-tree}, and
the proof of Theorem~\ref{thm-sc-cdim-lower} is completed in
Section~\ref{sec-lower-cdim}.

\subsection{Acknowledgments}

I would like to thank Ilya Kapovich for introducing me to some of the questions considered
in this paper.  
I also thank Piotr Przytycki for interesting conversations,
and the referee for many helpful suggestions.


\section{Random groups and small cancellation}\label{sec-rnd-grp-sc}

Our goal in this section is to study subwords of random groups in
the few relator model and density model.  
We find out what lengths subwords should be to be unique in the presentation,
or, on the other hand, so that every possible subword of that length appears.
These calculations are fairly routine, with some small technicalities from working with
cyclically reduced words as opposed to just reduced words.

We recall the definition of the metric small cancellation condition~\cite{Lyndon-Schupp-small-canc}.

\begin{definition}\label{def-piece-sc}
	The presentation $G = \langle S \mid R \rangle$ satisfies the metric
	small cancellation condition $C'(\lambda)$, for some $0<\lambda<1$,
	if every piece $u$ which is a subword of some cyclic conjugate of $r^{\pm 1}$,
	$r \in R$, satisfies $|u| < \lambda|r|$.
	A \emph{piece} is a common initial segment of two
	distinct cyclic conjugates of $r_1, r_2 \in R \cup R^{-1}$,
	where $r_1$ may equal $r_2$.
\end{definition}
%
%
%

\subsection{Small cancellation in the few relator model}
We have $m \geq 2, n \geq 1$ fixed.
Our goal in this subsection is to show that generic few relator presentations
satisfy strong small cancellation properties.

\begin{proposition}\label{prop-lambda-gen-upper}
	There exists $0 < C_0 < \infty$, depending only on $m$, so that 
	generic few relator presentations are $C'(\lambda_0(l))$,
	where $\lambda_0(l) = C_0 \frac{\log l}{l}$.
	In fact, we can take $C_0 = 11/\log(2m-1)$.
\end{proposition}
This result is essentially sharp, as shown by Proposition~\ref{prop-few-rel-wordlength}.

We begin with some preliminary observations.
In the following, the notation $A \asymp B$ indicates that $A \lesssim B \lesssim A$.

Let $N_l$ be the number of cyclically reduced words of length $l$ in $F_m$.
It is easy to see that $N_l \asymp (2m-1)^l$, with multiplicative error of $\frac{4}{3}$.
More precise estimates are in Subsection~\ref{ssec-cyc-reduced} below.
%
Let $N_{\leq l}$ be the number of cyclically reduced words of length at most $l$ in $F_m$.
Again, $N_{\leq l} \asymp (2m-1)^l$.
The number of presentations where all relators have length
at most $l$ is $N_{\leq l}^n = (N_{\leq l})^n$.  

Let $N_{[0.99l,l]}^n$ be the number of presentations where all $n$ relators
have length at least $0.99l$, but no more than $l$.  This is generic, since
\begin{align*}
	\frac{N_{\leq l}^n - N_{[0.99l,l]}^n}{N_{\leq l}^n} 
		\leq \frac{n \cdot N_{\leq 0.99l} \cdot N_{\leq l}^{n-1} }{N_{\leq l}^n} 
		\lesssim (2m-1)^{-0.01l},
\end{align*}
which goes to zero as $l \ra \infty$.

So to show that a property is generic, it suffices to show that it is generic within the
class of presentations where all relators have lengths between $0.99l$ and $l$.
\begin{proof}[Proof of Proposition~\ref{prop-lambda-gen-upper}]
	Let $N_{(l_i)}$ be the number of presentations with cyclically reduced 
	relators of length $|r_i| = l_i$, $i=1,\ldots,n$,
	and let $N_{(l_i),\lambda_0}^c$ be the number of those which are not $C'(\lambda_0)$.	
	It suffices to find an $o(1)$ bound for $N_{(l_i),\lambda_0}^c / N_{(l_i)}$,
	when the relators have lengths $l_1, \ldots, l_n$ in $[0.99l,l]$.
	
	If we fail to be $C'(\lambda_0)$, then there is a word $u$ of length 
	equal to $\lceil 0.99l\lambda_0 \rceil$ which appears
	in two distinct places in the words $r_1, \ldots, r_n$, or their inverses.
		
	\vspace{2mm}
	{\noindent\textbf{Case 1:}} The word $u$ appears in two different words.
	
	There are $\binom{2n}{2} \leq 4n^2$ choices for the words $r_i^{\pm 1}$ and $r_{i'}^{\pm 1}$.
	Given this choice, the number of ways $u$ can appear is bounded from above by
	the product of the number of choices of 
	(1) the location of $u$ in these words, 
	(2) the word $u$, 
	(3) the remainder of the words $r_i$ and $r_{i'}$,
	and (4) the other words.  
	Call these numbers $A_1, A_2, A_3$ and $A_4$ respectively. Clearly,
	\begin{gather*}
		A_1 \leq l^2, \quad
		A_2 \leq \frac{4}{3} (2m-1)^{|u|}, \quad \\
		A_3 \leq (2m-1)^{l_i-|u|} \cdot (2m-1)^{l_{i'}-|u|}, \quad \text{and} \quad
		A_4 = \prod_{j \neq i,i'} N_{l_j}.
	\end{gather*}
	Since we have
	\begin{align*}
		\frac{A_1 A_2 A_3 A_4}
					{ \prod_{j = 1,\ldots,n} N_{l_j} }
			& \lesssim \frac{l^2 (2m-1)^{|u|}(2m-1)^{l_i-|u|}(2m-1)^{l_{i'}-|u|}} 
					{ N_{l_i} \cdot N_{l_{i'}} } \\
			& \lesssim l^2 (2m-1)^{-|u|},
	\end{align*}
	Case 1 occurs with probability $P_1$ at most
	$ P_1  \lesssim n^2 l^2 (2m-1)^{-|u|} $.
	Observe that $-|u| \leq -0.99 C_0 \log(l)$, and $l^2 = (2m-1)^{2\log(l)/\log(2m-1)}$.
	Therefore, provided $2-0.99 C_0 \log(2m-1) < 0$,
	the probability $P_1$ will go to zero as $l$ goes to infinity.
	
	\vspace{2mm}
	{\noindent\textbf{Case 2:}} The word $u$ appears in the same word $r_i$ in two distinct ways.
	
	Let $P_2$ be the probability this occurs among presentations of lengths
	$(l_i)$.
	
	\begin{lemma}
		There is a subword $v$ of $u$, of length at least $0.2C_0 \log(l)$,
	which appears in $r_i$ in two non-intersecting locations as either $v$ or $v^{-1}$.
	\end{lemma}
	\begin{proof}
		Consider $r_i$ as a labelling on the oriented circle.
		Let $u_1$ and $u_2$ be the two words on the boundary $r_i$ so that each is 
		labelled by $u$ or $u^{-1}$.
	
		If the initial segment of $u_1$ of length $\lceil 0.2 C_0 \log(l) \rceil$
		does not intersect $u_2$, then let $v$ be that subword, and we are done.
		
		Otherwise, up to relabelling $u_1$ and $u_2$, 
		we can assume that the initial letter of $u_1$ is not in $u_2$
		but that the initial segment of $u_1$ of length $\lceil 0.2 C_0 \log(l) \rceil$
		does meet $u_2$.
		
		If the word $u$ has opposite orientations in $u_1$ and $u_2$, we let $v$ be the initial
		segment of $u_1$ of length $\lceil 0.2 C_0 \log(l) \rceil$. 
		Then $v^{-1}$ also appears in the tail segment of $u_2$, disjoint from $v$.
		
		Finally, if $u$ has the same orientation in both $u_1$ and $u_2$, let $w$ be the initial
		segment of $u_1$ disjoint from $u_2$, of length at most $0.2 C_0 \log(l)$.
		Since the words $u_1$ and $u_2$ are both copies of $u$, $u$ is made up of repeated
		copies of $w$ followed by some tail $w'$.
		We write $u = w^{2k} w'$, for some integer $k$, and word $w'$ of length 
		$|w'| < 2 \lceil 0.2 C_0\log(l) \rceil$, thus
		$|w^k| \geq 0.2 C_0 \log(l)$, so $v=w^k$ is our required word.  
		(In some of these estimates we assumed that $l$ was sufficiently large.)		
	\end{proof}
	
	We can now find, analogous to Case 1, that
	$P_2$, is bounded from above by
	the product of the number of choices of $i$,
	the locations of $v$ in this word, 
	the word $v$, 
	the remainder of the word $r_i$,
	all divided by $N_{l_i}$.  Therefore
	\begin{equation*}
		P_2  \lesssim \frac{n \cdot l^2 \cdot (2m-1)^{|v|} \cdot (2m-1)^{l_i-2|v|}}
			{N_{l_i}}
			 \lesssim  l^2 (2m-1)^{- |v|}.
	\end{equation*}
	Now $-|v| \leq -0.2 C_0 \log(l)$, so provided
	$2-0.2 C_0 \log(2m-1) < 0$,
	the probability $P_2$ will go to zero as $l$ goes to infinity.
	
	\vspace{2mm}
	{\noindent\textbf{Combining the cases:}}
	
	We have shown that $N_{(l_i),\lambda_0}^c/N_{(l_i)} \leq P_1 + P_2$
	goes to zero as $l \ra \infty$, independent of the choice of $l_i$ between $0.99l$ and $l$,
	provided that $C_0$ is sufficiently large.  It suffices to take
	$
		C_0 = 11/\log(2m-1).
	$
\end{proof}

\subsection{Counting cyclically reduced words}\label{ssec-cyc-reduced}
In this subsection we give some lemmas we will use in the remainder of this section.
We will need the following lemma which counts the number of ways to fill in a cyclically reduced word.
\begin{lemma}\label{lem-cyc-red-word-count}
	We count all reduced words $w$ of length $n+2$
	with first and last letter fixed in the free group $\langle s_1, s_2, \ldots, s_m \rangle$.
	
	There are essentially three different cases.
	Let $p_n$, $q_n$ and $r_n$ count the number of reduced words of length $n+2$ of the forms
	$s_1 u s_1$, $s_1 u s_1^{-1}$ and
	$s_1 u s_2$, respectively.
	Then, for all $n \geq 1$, we have:
	\[
		\frac{\max\{p_n,q_n,r_n\}}{\min\{p_n,q_n,r_n\}} \leq 
			1+ \frac{2}{(2m-1)^n}.
	\]
\end{lemma}
\begin{proof}
	Note that $p_1 = 2m-1$, and $q_1=r_1=2m-2$.
	Clearly,
	\begin{align*}
		p_n &= p_{n-1}+(2m-2)r_{n-1}, \\
		q_n &= q_{n-1}+(2m-2)r_{n-1}, \ \text{and} \\
		r_n &= p_{n-1}+q_{n-1}+(2m-3)r_{n-1}. 
	\end{align*}
	One observes that, by induction,
	when $n$ is odd, $p_n = q_n +1$ and $r_n = q_n$,
	while when $n$ is even, $p_n = r_n = q_n + 1$.
	
	A simple recurrence relation calculation gives that
	\begin{equation*}
		q_n = 
		\begin{cases}
			\frac{1}{2m} \big( (2m-1)^{n+1} - 1 \big) & \text{if $n$ is odd},\\
			\frac{1}{2m} \big( (2m-1)^{n+1} - (2m-1) \big) & \text{if $n$ is even}.
		\end{cases}
	\end{equation*}
	Therefore
	\begin{gather*}
		q_n \geq \frac{2m-1}{2m} \big( (2m-1)^n - 1 \big) \geq \frac{1}{2} (2m-1) ^n, 
	\end{gather*}
	and
	\[
		\frac{\max\{p_n,q_n,r_n\}}{\min\{p_n,q_n,r_n\}}
		= \frac{q_n+1}{q_n}
		= 1 + \frac{1}{q_n}
		\leq 1 + \frac{2}{(2m-1)^n}. \qedhere
	\]
\end{proof}
This proof implies that $N_l = 2mp_{l-1} \asymp (2m-1)^l$.

The following lemma estimates the probability of omitting a specified word.
\begin{lemma}\label{lem-cyc-omit-gl-word}
	Fix a reduced word $r_0$ of length $g(l) < l/4$, $g(l) > 4$.
	Let $N_{r_0}$ be the number of all cyclically reduced words of length
	$l$ which omit $r_0$.
	Then the proportion $N_{r_0}/N_l$ is at most
	\[
		\frac{N_{r_0}}{N_l} \leq 
			\exp \left(\frac{2}{(2m-1)^{(l/2)-1}} - 
			\frac{l}{9 g(l) (2m-1)^{g(l)}}
			\right).
	\]
\end{lemma}
\begin{proof}
	Consider a cyclically reduced relator $r_1$ of length $l$ which omits $r_0$.
	Let $A = \lfloor \frac{l}{2(g(l)+1)} \rfloor$.
	Let us split up $r_1$ into an initial letter, then
	words $u_1, u_2, \ldots, u_A$ of length $g(l)+1$, plus a tail of length $t$, 
	where $t$ must be between $(l/2) - 1$ and $3l/4$.
	Each word $u_i$ consists of an initial letter, plus a word of length $g(l)$,
	which is not $r_0$.
	
	The initial letter of $r_1$ has $2m$ possibilities.
	For each $i=1,\ldots, A$, either the initial letter of $u_i$ matches the inverse of 
	the initial letter of $r_0$, or it does not.
	In the former case, the remaining $g(l)$ letters have $(2m-1)^{g(l)}$ possibilities,
	while in the latter case there are only $(2m-1)^{g(l)}-1$ possibilities, since
	the word $r_0$ is excluded.  
	The number of possibilities for the remaining $t$ letters is
	bounded by $\max\{p_t,q_t,r_t\}$ (as defined in Lemma~\ref{lem-cyc-red-word-count}).
	Altogether, we have a bound
	{\allowdisplaybreaks
	\begin{align*}
		\frac{N_{r_0}}{N_l}
		& \leq \frac{2m \left( (2m-1)^{g(l)} + (2m-2) \left[ (2m-1)^{g(l)}-1 \right]
			\right)^A \max\{p_t,q_t,r_t\}}
			{2m (2m-1)^{(g(l)+1)A} \min\{p_t,q_t,r_t\}} \\
		& \leq \left( \frac{(2m-1)^{g(l)} + (2m-2) \left[ (2m-1)^{g(l)}-1 \right]}
			{(2m-1)^{g(l)+1}} \right)^{\! A} \! \! \left( 1+ \frac{2}{(2m-1)^t} \right) \\
		& =  \left( 1 - \frac{(2m-2)}{(2m-1)^{g(l)+1}} \right)^A 
			\left( 1+ \frac{2}{(2m-1)^t} \right) \\
		& \leq \exp \left(\frac{2}{(2m-1)^t} - A \cdot \frac{(2m-2)}{(2m-1)^{g(l)+1}} \right),
			\text{using}\ 1+x\leq e^{x}.
	\end{align*} }
	Observe that $A \geq \frac{l}{6g(l)}$, and $\frac{2m-2}{2m-1}\geq \frac{2}{3}$, thus:
	\begin{equation*}
		\frac{N_{r_0}}{N_l}
		\leq \exp \left(\frac{2}{(2m-1)^{(l/2)-1}} - 
			\frac{l}{9 g(l) (2m-1)^{g(l)}}
			\right).\qedhere
	\end{equation*}
\end{proof}

\subsection{Short subwords of generic few relator presentations}
\begin{proposition}\label{prop-few-rel-wordlength}
	There exists a constant $C$ (depending on $m$) so that a generic few relator
	presentation with relator lengths at most $l$ 
	contains every reduced word of length $\lceil C \log(l) \rceil$
	as a subword of some relator.
	
	In fact, we can take any $C < 1/\log(2m-1)$.
\end{proposition}
We will actually show that every reduced word of length $\lceil C \log(l) \rceil$
appears as a subword of every relator.
\begin{proof}
	Let $N_{g(l)}$ be the number of cyclically reduced words in $\langle S \rangle$
	of length $l$ which contain every word of length at most $g(l)$.
	To prove the proposition, it suffices to show that
	$(N_l - N_{g(l)})/N_l \ra 0$ as $l \ra \infty$, where
	$g(l) = \lceil C \log(l) \rceil$.

	By Lemma~\ref{lem-cyc-omit-gl-word}, the probability of an individual relator
	omitting a fixed word $r_0$ of length $g(l)$ is at most
	\[
		\exp \left(1 - \frac{l}{9 g(l) (2m-1)^{g(l)}} \right).
	\]
	There are at most $\frac{4}{3}(2m-1)^{g(l)}$ choices for $r_0$, so the 
	probability of missing some word of length $g(l)$ satisfies
	\begin{align*}
		\frac{N_l - N_{g(l)}}{N_l}
			&\leq \frac{4}{3}(2m-1)^{g(l)} \cdot
				\exp \left(1 - \frac{l}{9 g(l) (2m-1)^{g(l)}} \right) \\
			&\leq 4 \exp \left( \log(2m-1)g(l) - 
				\frac{l}{9 g(l) (2m-1)^{g(l)}} \right).
	\end{align*}
	Note that since $g(l) = \lceil C \log(l) \rceil$, 
	$(2m-1)^{g(l)}$ behaves like $l^{C\log(2m-1)}$ for large $l$.
	Thus, if $C \log(2m-1) < 1$, then
	$\frac{N_l - N_{g(l)}}{N_l}$ will go to zero as $l \ra \infty$.	
\end{proof}

\subsection{Short subwords in the density model}

The following proposition is a version of \cite[Prop.~9]{Oll-05-rand-grp-survey}.
Ollivier sketches a proof for $0<C<d<1$; for completeness we provide a proof in the
following special case.
\begin{proposition}\label{prop-dens-wordlength}
	For any $0<C < d< 1/4$, a generic presentation
	at density $d$ contains every reduced word of length $\lceil C l \rceil$ as a subword
	of some relator.
\end{proposition}

\begin{proof}
	This follows a similar proof to Proposition~\ref{prop-few-rel-wordlength}.
	There are $(2m-1)^{dl}$ reduced words chosen independently, so the probability
	that they all omit a particular word $r_0$ of length $g(l) = \lceil C l \rceil$
	is, by Lemma~\ref{lem-cyc-omit-gl-word}, at most
	\begin{align*}
		& \left[ 
			\exp \left(\frac{2}{(2m-1)^{(l/2)-1}} -
					\frac{l}{9 g(l) (2m-1)^{g(l)}} \right)
			\right]^{(2m-1)^{dl}} \\
		& \quad  = 
			\exp \left(\frac{2 (2m-1)^{dl}}{(2m-1)^{(l/2)-1}} - 
					\frac{l(2m-1)^{dl}}{9 g(l) (2m-1)^{g(l)}} \right) \\
		& \quad \lesssim \exp \left( \frac{-1}{10C}(2m-1)^{(d-C)l-1}
			\right),
	\end{align*}
	for sufficiently large $l$.
	
	Again, there are at most $\frac{4}{3}(2m-1)^{g(l)}$ choices for $r_0$,
	so the probability $P$ that some word of length $g(l) = \lceil C l \rceil$
	is omitted satisfies
	\begin{align*}
		P & \lesssim \frac{4}{3}(2m-1)^{g(l)} \cdot
			\exp \left( \frac{-1}{10C}(2m-1)^{(d-C)l-1}
			\right) \\
		& \lesssim
			\exp \left( 2\log(2m-1)Cl - 
				\frac{1}{10C}(2m-1)^{(d-C)l-1}
				\right),
	\end{align*}
	for large $l$, and this goes to zero as $l \ra \infty$.
\end{proof}


\section{Cayley graphs of small cancellation groups}\label{sec-sc-cayley}

In $C'(1/6)$ small cancellation groups, geodesic bigons and triangles
are known to have certain special forms \cite{Str-88-sc-hyp,Cha-95-rand-grps}.
In this section we recall these standard facts, and give some extensions
to the case of geodesic $n$-gons which will be needed in Section~\ref{sec-ac-upper-cdim}.

Throughout this section, $G = \langle S \mid R \rangle$ is a finitely presented group,
where every $r \in R$ is a cyclically reduced word in $\langle S \rangle$.

\begin{definition}\label{def-diagram}
	A \emph{diagram} for a reduced word $w \in G$ is a connected,
	contractible, finite, pointed,
	planar 2-complex $\cD$ which satisfies the following conditions:
	\begin{enumerate}
		\item Each edge of $\cD$ is oriented and labelled with an element of $S$,
		\item For each face $B \subset \cD$, reading the edge labels along
			its boundary $\partial B$ gives a (cyclic conjugate of) a word 
			$r^{\pm 1}$, $r \in R$.
		\item The base point lies on the boundary $\partial\cD$, and reading
			the edge labels from this point around $\partial\cD$ 
			counter-clockwise gives $w$.
	\end{enumerate}
	We say $\cD$ is \emph{reduced} if there are never two distinct
	faces $B_1,B_2$ which
	intersect in at least one edge, so that the labellings on $\partial B_1$ and
	$\partial B_2$, read from this edge clockwise and counter-clockwise respectively,
	agree.
\end{definition}

\begin{lemma}[Strebel \cite{Str-88-sc-hyp}]\label{lem-sixfb}
	Suppose $\cD$ is a reduced diagram homeomorphic to a disc.
	For a vertex $v$, let $d(v)$ denote its degree.
	For a face $B$, let $|\partial B|$ denote its degree,
	let $e(B)$ denote the number of exterior edges of $B$,
	and let $i(B)$ denote the number of interior edges.
	Then
	\begin{equation}
		6 = 2 \sum_v(3 - d(v)) + \sum_B (6 - 2e(B) -i(B)).\label{eq-sixfb}
	\end{equation}
\end{lemma}
\begin{proof}
	Suppose $\cD$ has $V$ vertices, $E$ edges and $F$ faces.  Then
	\begin{align}
		1 + E & = V + F = \sum_v 1 + \sum_B 1, \label{eq-vef1}\\
		2E & = \sum_v d(v) \label{eq-vef2}, \text{ and}\\
		2E & = \bigg( \sum_B |\partial B| \bigg) + |\partial\cD| = \sum_B \left( 2e(B)+i(B) \right) \label{eq-vef3}.
	\end{align}
	Consider $6\cdot\eqref{eq-vef1}-2\cdot\eqref{eq-vef2}-\eqref{eq-vef3}$.
\end{proof}

\begin{definition}
	The Cayley graph $\Gamma(G,S) = \Gamma^1(G,S)$ of a group $G$ with finite
	generating set $S$ is the graph with vertex set $G$,
	and an unoriented edge between $\{g, gs\}$ for all $g \in G$, $s \in S \cup S^{-1}$.
\end{definition}

Suppose $P$ is a geodesic $n$-gon in the Cayley graph $\Gamma(G,S)$,
where $G=\langle S \mid R \rangle$ satisfies $C'(\lambda)$, for
some $\lambda\in(0, \frac{1}{6}]$.
We want to show that $P$ is slim; that is, any side of $P$
is contained in a suitable neighborhood of the other sides.

As $P$ is a closed loop, van Kampen's lemma states that there is a
reduced diagram $\cD$ for $P$.  We may assume that the boundary word
is cyclically reduced, and that $\cD$ is homeomorphic to a disc; this
only makes it harder to show that $P$ is slim.
  
We remove all vertices of degree $2$ from $\cD$ and 
relabel edges with the corresponding words in $\langle S \rangle$.  So now all vertices
have degree at least $3$.

In this reduced diagram, there are two kinds of faces that have external
edges, those where a endpoint of a side of $P$ lies in the interior
of an external edge, and all others.  We call the former kind distinguished;
there are at most $n$ of them.

When $e(B)=1$ and $B$ is not distinguished, the external edge with label
$u$ is a geodesic in $\Gamma(G,S)$, and so $|u| \leq \frac{1}{2} |\partial B|$.
Now each remaining edge of $B$ is internal, and so a piece of $G$, and so
has length less than $\lambda |\partial B|$.
Thus
\[
	\frac{1}{2} |\partial B| \leq \sum \big\lbrace |t| : t \ \text{internal edge of } B \big\rbrace
		< i(B)\lambda |\partial B|,
\]
So $i(B) > \frac{1}{2\lambda}$, thus 
$i(B) \geq \lfloor \frac{1}{2\lambda} + 1 \rfloor =: d_{Ext}(\lambda) \geq 4$.

Note also that each edge of an interior face $B$ ($e(B)=0$) has 
length strictly less than $\lambda|\partial B|$, so
\[
	i(B) > \frac{1}{\lambda} \quad\Rightarrow\quad i(B) \geq 
		\Big\lfloor \frac{1}{\lambda} + 1 \Big\rfloor =: d_{Int}(\lambda) \geq 7.
\]
Thus \eqref{eq-sixfb} splits into cases as follows.
\begin{align}
	\label{eq-sixfbcases}
	\begin{split}
	6 & = 2 \sum_v(3 - d(v)) + \sum_{B,\,e(B)=0} (6 -i(B)) +
		\sum_{\substack{B,\,e(B)=1\\ \text{not dist.}}} (4 -i(B)) + \\ & \quad
		\sum_{\substack{B,\,e(B)=1\\ \text{dist.}}} (4 -i(B)) +  
		\sum_{B,\,e(B)=k\geq 2} (6 - 2k -i(B))
	\end{split}\\
	& \leq - (d_{Int}(\lambda)-6) F_I + 3n,\notag
\end{align}
where $F_I$ is the number of interior faces of $\cD$.
We have shown the following.
\begin{lemma}\label{lem-int-face-bound}
	In the above situation,
	\[
		F_I \leq \frac{3n-6}{d_{Int}(\lambda)-6}.
	\]
\end{lemma}

We now consider other aspects of the geometry of $C'(1/6)$ groups
that will be needed in the remainder of the paper.
\begin{lemma}\label{lem-onesixth-nice-faces}
	Suppose $G=\langle S | R\rangle$ is $C'(1/6)$,
	and that the diagram $\cD$ has no vertices of degree two.
	Then any two distinct faces $B,B' \subset \cD$ are
	either disjoint, meet at a single point, 
	or meet along a single edge.
	
	Also, the boundary of any face $B$ is a simple curve,
	i.e., the face does not bump into itself.
%
\end{lemma}
\begin{proof}
%
	If the boundary of a face $B$ is not a simple curve,
	$B$ encloses a subdiagram $\cD'$ 
	in the interior of $\cD$, all of whose vertices (except
	perhaps one) have degree at least three, and all of
	whose faces have degree at least seven.
	This contradicts \eqref{eq-sixfb}.
	
	Similarly, if two faces meet at more than a single
	edge, they enclose a subdiagram $\cD'$ in the interior
	of $\cD$, and this has at most two vertices of degree two.
	This again contradicts \eqref{eq-sixfb}.
\end{proof}

Lemma~\ref{lem-int-face-bound} immediately implies that reduced diagrams for
geodesic bigons have no internal faces, and that reduced diagrams for
geodesic triangles have at most three internal faces.
We can make more precise statements in these cases.
See \cite[Theorem 35]{Str-88-sc-hyp} and \cite[Proposition 3.6]{Cha-95-rand-grps} for proofs.
\begin{lemma}\label{lem-thin-triangles}
	Reduced diagrams for geodesic bigons in a $C'(1/6)$ group have a specific
	form, as illustrated by Figure~\ref{fig-bigon}.  
	\begin{figure}
		\begin{center}
		\includegraphics[width=0.9\textwidth]{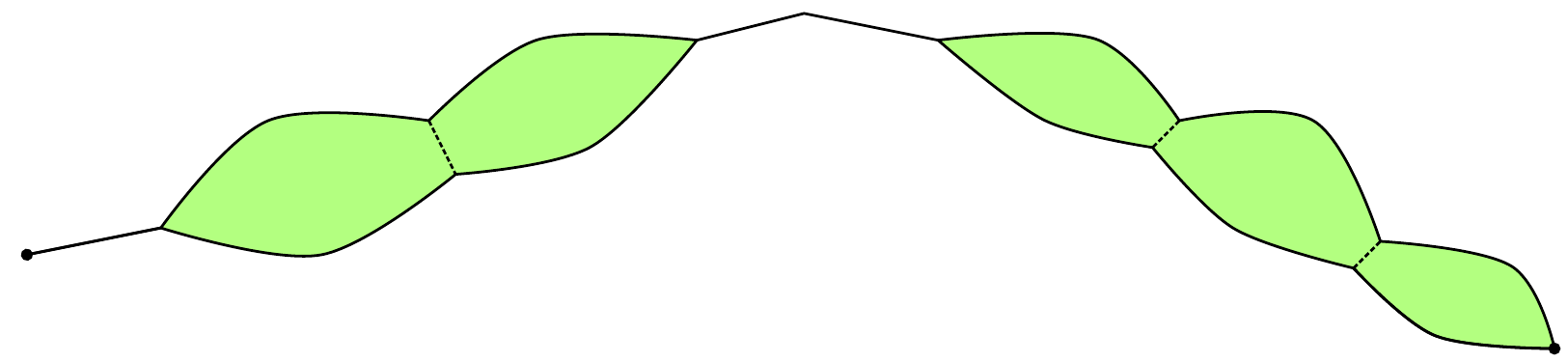}
		\end{center}
		\caption{A geodesic bigon in a $C'(1/6)$ group}\label{fig-bigon}
	\end{figure}
	
	Reduced diagrams for geodesic triangles in a $C'(1/6)$ group have
	no interior faces.   
	In particular, the Cayley graph is $2M$-hyperbolic,
	where $M = \max_{r \in R} |R|$.
	After removing spurs, the reduced diagram for a geodesic triangles has no more than 
	six connected faces.
	(Figure~\ref{fig-tri} gives an example, with the dual diagram of these six faces indicated.)
\end{lemma}
Recall that if a geodesic triangle has sides 
$\gamma_{12}, \gamma_{13}$ and $\gamma_{23}$ joining
vertices $P_1, P_2$ and $P_3$, and $\cD$ is a reduced diagram for the triangle,
then the spur of $\cD$ containing $P_1$ is the maximal
subdiagram of $\cD$ bounded by $\gamma_{12}, \gamma_{13}$ and a vertex or a single internal edge.

In Sections \ref{sec-round-tree} and \ref{sec-lower-cdim}, we will use the 
following slight generalization of the geodesic bigon description above.
\begin{lemma}\label{lem-bigon-face}
	Suppose $\cD$ is a reduced diagram in a $C'(1/6)$ group
	whose boundary is labelled by, in order, a geodesic $[p,u]$, part of a face
	$B \subset \cD$, and a geodesic $[v,p]$.
	Then $\cD$ has the same form as a diagram for a bigon as illustrated by Figure \ref{fig-bigon} above.
\end{lemma}
\begin{proof}
	As in the calculation of \eqref{eq-sixfbcases}, we remove all degree two vertices from $\cD$.
	We can assume that $[p,u] \cap [p,v] = \{p\}$ in $\cD$, otherwise a geodesic bigon is formed and
	we can remove it and continue.  Thus we have a diagram homeomorphic to a disc.
	Every internal face has at least seven edges.
	Every face with an external edge, with the possible exceptions of $B$ and the face containing $p$
	in its boundary, will have at least four internal edges.
	
	Therefore by \eqref{eq-sixfbcases}, there are no internal faces.
	Moreover, the inequalities in \eqref{eq-sixfbcases} are equalities,
	and so $B$ and the face containing $p$ have exactly one internal edge,
	and all vertices have degree three.
	Also, all faces with at least two external edges have exactly two
	external and two internal edges.
	Thus the faces adjacent to $B$ and the face containing $p$
	have exactly two external and two internal edges.
	We continue, and deduce that the diagram has the form
	of a chain of faces from $p$ to $B$ meeting along single internal edges.
\end{proof}

	\begin{figure}
		\begin{center}
		\includegraphics[width=0.7\textwidth]{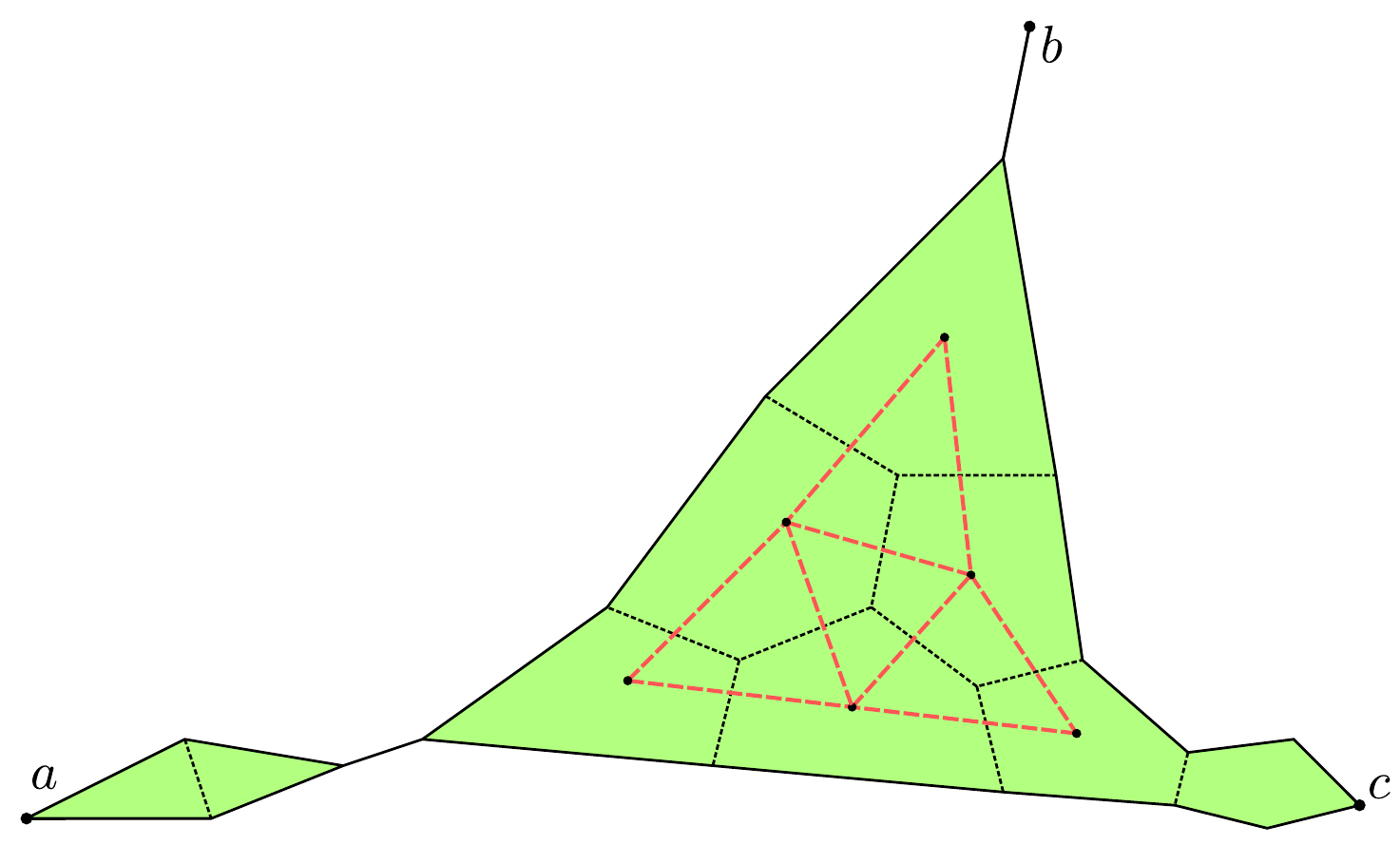}
		\end{center}
		\caption{A geodesic triangle in a $C'(1/6)$ group}\label{fig-tri}
	\end{figure}


\section{Asymptotic curvature bounds and an upper bound for conformal dimension}\label{sec-ac-upper-cdim}

\subsection{Outline}
If $G$ is hyperbolic, geodesic triangles in $\Gamma(G,S)$ are uniformly slim.
Consequently, geodesic $n$-gons will be $(C\log(n))$-slim, for some $C$ independent
of $n$.
Bonk and Foertsch \cite{BF-06-ACu-kappa} investigated this further and linked
the behavior of geodesic $n$-gons to the optimal visual parameter $\epsilon$ for
visual metrics on the boundary of $G$.
In this section we will use these ideas to prove the following theorem.
\begin{theorem}\label{thm-sc-cdim-upper}
	If $G = \langle S | R \rangle$ is a $C'(\lambda)$ presentation of a group, with
	$\lambda \leq {1}/{6}$, and $|r| \leq M$ for all $r \in R$, then
	\[
		\Cdim(\partial_\infty(G)) 
			\leq \frac{M}{2\log\lfloor\frac{1}{\lambda}-4\rfloor} \log(2|S|-1).
	\]
\end{theorem}

We recall one of the equivalent definitions of asymptotic upper curvature,
and the result which we will need.
\begin{definition}[Bonk and Foertsch]\label{def-AC-upper}
	A geodesic metric space $X$ has
	\emph{an asymptotic upper curvature bound $\kappa$},
	written $\AC_u(\kappa)$, for $\kappa \in [-\infty,0)$,
	if there exists some $C$ so that for every $n \in \N$, $n \geq 2$,
	every geodesic $(n+1)$-gon in $X$
	is $\left(\frac{1}{\sqrt{-\kappa}}\log(n) + C\right)$-slim.
	
	(Recall that a geodesic $(n+1)$-gon is $\Delta$-slim if every side is in the union of
	the $\Delta$-neighborhoods of the other $n$ sides.)
\end{definition}

\begin{theorem}[{\cite[Theorem 1.5]{BF-06-ACu-kappa}}]\label{thm-AC-visual}
	If a geodesic metric space $X$ is \linebreak $\AC_u(\kappa)$, for some $\kappa \in [-\infty,0)$, then
	for every $0<\epsilon<\sqrt{-\kappa}$ there is a visual metric on $\partial_\infty X$
	with parameter $\epsilon$.
\end{theorem}

This result, and the bound in \eqref{eq-cdim-eps-entropy},
reduce the proof of Theorem~\ref{thm-sc-cdim-upper} to the following statement,
which we prove in the following subsection.
\begin{theorem}\label{thm-sc-ac-link}
	If $G = \langle S | R \rangle$ is a $C'(\lambda)$ presentation of a group, with
	$\lambda \leq \frac{1}{6}$, and $|r| \leq M$ for all $r \in R$, then
	the Cayley graph $\Gamma(G,S)$ is $AC_u(\kappa)$ with
	$\kappa = -\frac{4}{M^2}\log^2\big\lfloor\frac{1}{\lambda}-4\big\rfloor$.
\end{theorem}

Note that every $\delta$-hyperbolic space is $AC_u(\kappa)$ for some 
$\kappa \approx -\frac1{\delta^2}$, see \cite[Equation (3)]{BF-06-ACu-kappa}.
The groups considered in Theorem~\ref{thm-sc-ac-link} all have $\delta = 2M$,
but we can use the $C'(\lambda)$ condition when $\lambda$ is small to find an 
improved $\kappa$ bound.
%
%
%

\subsection{Slim \textit{n}-gons}
To prove Theorem~\ref{thm-sc-ac-link}, we show that while a reduced diagram
for a geodesic $n$-gon may have interior faces, they cannot
be too far from the boundary of the diagram.
\begin{proposition}\label{prop-facesthin}
	Let $\cP$ be a geodesic $n$-gon in the Cayley graph of a 
	$C'(\lambda)$ group $G = \langle S | R \rangle$, $\lambda \leq 1/6$, 
	and let $\cD$ be a reduced diagram for $\cP$.
	
	Then there exists some constant $C=C(\lambda)$ so that,
	for any $x \in \partial\cD = \cP$, there is a chain of at most
	$k+1$ faces joining $x$ to another side of $\cP$, with
	\[
		0 \leq k \leq \frac{\log(n)}{\log \lfloor \tfrac{1}{\lambda}-4 \rfloor} + C.
	\]
	This means that there are faces $B_0,B_1,\ldots, B_k$ so that
	$x \in \partial B_0$, $B_j \cap B_{j+1} \neq \emptyset$ for $0 \leq j < k$,
	and $B_k$ meets another side of $\cP$.
\end{proposition}

\begin{proof}
	We begin with the following lemma.
	\begin{lemma}\label{lem-geodesicloop}
		Let $G$ be a $C'(1/6)$ group.
		If $\cD$ is a reduced diagram containing a face $B$ and a geodesic segment $\gamma$,
		then $\gamma \cap B$ is connected.	
	\end{lemma}
	\begin{proof}
		Suppose not.  Then $\gamma$ and $B$ enclose a subdiagram $\cD'$ of $\cD$.
		As in Section~\ref{sec-sc-cayley}, we assume that we have removed
		all degree two vertices from $\cD'$.
		Every internal face of $\cD'$ has at least seven internal faces.
		Since $\gamma$ is a geodesic segment, every face with an external edge, except
		possibly $B$, will have at least four internal edges.
		
		Therefore \eqref{eq-sixfb} gives us the following contradiction:
		\[ 
			6 \leq (6-2e(B)-i(B)) \leq 3. \qedhere
		\]
	\end{proof}
	We now continue the proof of the proposition.
	
	The point $x$ lies in the boundary of some face $B_0 \subset \cP$.
	Let $\gamma \subset \cP$ be the geodesic side of $\cP$ containing $x$.
	We may assume that $B_0$ does not meet $\cP \setminus \gamma$, else the single chain $B_0$ suffices.
	
	Lemma~\ref{lem-geodesicloop} shows that $B_0$ has a single exterior edge that
	is a geodesic segment in $\gamma$.
	By Lemma~\ref{lem-onesixth-nice-faces}, $B_0$ is homeomorphic
	to a closed disc, with
	$i(B_0) \geq \lfloor \frac{1}{2\lambda} +1 \rfloor \geq 4$.
	
	Recall that the \emph{star neighborhood} of a subcomplex $\cD' \subset \cD$ is the union
	of all closed cells in $\cD$ meeting $\cD'$, and is denoted by $\St(\cD')$.	
	
	Let $\cD_0$ be the subdiagram consisting of the single face $B_0$.
	For $i \geq 0$, let $\cD_{i+1} = \St(\cD_i)$.
	
	We will show that the number of faces in $\cD_i$ grows exponentially in $i$ until another
	side of $\cP$ is found.
	An example of this is shown in Figure~\ref{fig-nbhd-growth}.
	\begin{figure}
		\begin{center}
		\includegraphics[width=0.9\textwidth]{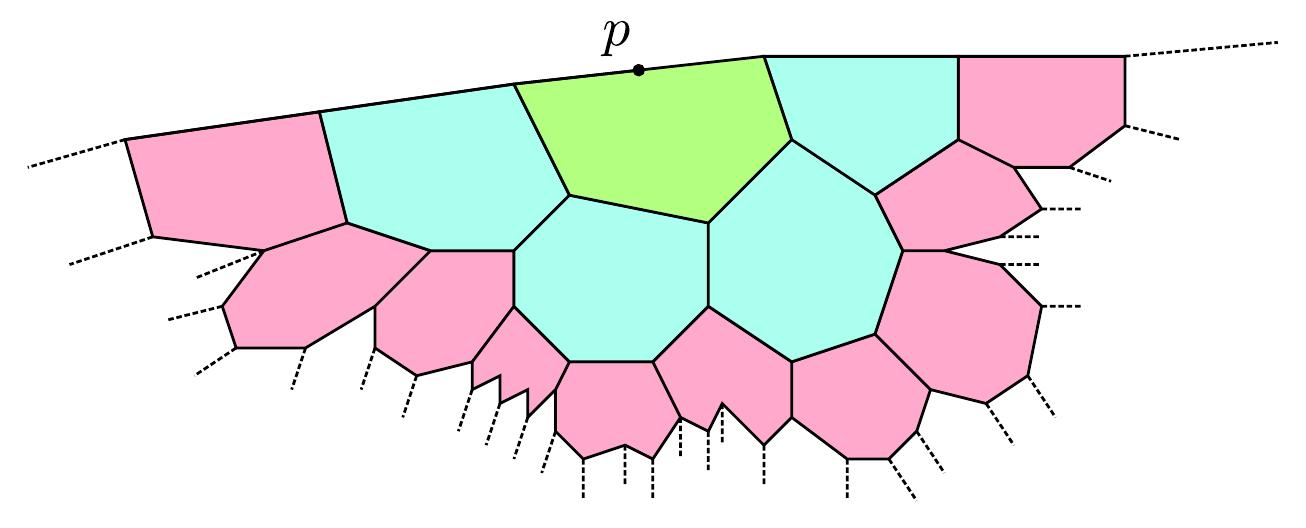}
		\end{center}
		\caption{Star neighborhoods of $B_0$}\label{fig-nbhd-growth}
	\end{figure}
	
%
%
	Suppose that $\cD_i$ is homeomorphic to a disc, which meets $\cP$ in a
	geodesic segment of $\gamma$, and that $\cD_{i+1}$ does not meet another side of $\cP$.
	We can extend the geodesic segment to edges on faces $B^{(i+1)}_-$ and $B^{(i+1)}_+$
	adjacent to $\cD_i$.
		
	Taking a small neighborhood $U$ of $\cD_i$ in $\St(\cD_i)$, there is a natural
	way to order the faces in $\St(\cD_i)\setminus\cD_i$ from
	$B^{(i+1)}_-=B^{(i+1)}_1$ to $B^{(i+1)}_{l_{i+1}} = B^{(i+1)}_+$, so that
	$B^{(i+1)}_p$ meets $B^{(i+1)}_q$ along an edge in $U\setminus\cD_i$ if and only if $|p-q|=1$.
	
	We will control the geometry of the diagrams $\cD_i$ using the following assumption.
	Let 
	\begin{itemize}
		\item[$A(i)$:] First, $\cD_i$ is homeomorphic to a closed disc which meets $\cP$ in a
		geodesic segment of $\gamma$.
			Second, $B^{(i)}_j$ has at least $d_{Ext}(\lambda)-2 \geq 2$ 
			edges in $\partial \cD_i$ that are internal in $\cD$, for $j=1$ and $j=l_i$,
			and $B^{(i)}_j$ has at least $d_{Int}(\lambda)-4\geq 3$ 
			internal (in $\cD$) edges in $\partial \cD_i$, for $1 < j < l_i$.
			Finally, there are at least four internal (in $\cD$) edges in $\partial \cD_i$.
	\end{itemize}
	We let $l_0 = 1$ and $B^{(0)}_1 = B_0$, and note that $A(0)$ holds.	
	
	The following lemma will provide the key induction step in our proof.
	\begin{lemma}
		Suppose that $A(i)$ holds.
		Then either $\cD_{i+1}$ meets another side of $\cP$, or $A(i+1)$ holds, with
		\[
			l_{i+1} \geq (l_i-2) \left\lfloor \tfrac1\lambda - 4 \right\rfloor.
		\]
	\end{lemma}
	\begin{proof}
		We assume that $\cD_{i+1}$ does not meet another side of $\cP$.
		Define $l_{i+1}$ and 
		$B^{(i+1)}_1, B^{(i+1)}_2, \ldots, B^{(i+1)}_l$
		as above.
		
		Let $\cD / \cD_i$ be the diagram formed from $\cD$ by combining $\cD_i$ into a single
		face $B_{\cD_i} \subset \cD/\cD_i$, and removing all vertices of 
		degree two.
		Even though this diagram has a face not labelled by a relator, it and its subdiagrams 
		will be helpful in the rest of the proof.
		
%
		
		\vspace{2mm}
		\noindent\textbf{Claim 1:} If $p\neq q$ then $B^{(i+1)}_p \neq B^{(i+1)}_q$.
		
		Suppose we have $B^{(i+1)}_p = B^{(i+1)}_q$, $p \neq q$. 
		We can assume that $q \geq p+2$, since Lemma~\ref{lem-onesixth-nice-faces}
		states that faces don't bump into themselves,
		and that $q-p$ is minimal.
		Let $\cD' \subset \cD/\cD_i$ be the subdiagram which includes
		$B_{\cD_i}$ and $B^{(i+1)}_p = B^{(i+1)}_q$ and all they enclose, 
		and remove all degree two vertices from it.
		
		This may reduce the number of edges of an interior face of $\cD'$
		below seven by removing a vertex at the junction of some
		$B^{(i)}_j$ and $B^{(i)}_{j+1}$ in $\partial \cD_i$.
		By assumption $A(i)$, these junctions are seperated by several
		edges in $\partial \cD_i$.
		So by the minimality of $q-p$, and construction of the diagrams $\cD_i$,
		no interior face of $\cD'$ can meet two of these junctions,
		and thus every interior face of $\cD'$ has at least six edges.
		
		Both $B_{\cD_i}$ and $B^{(i+1)}_p$ have at least one external edge in $\cD'$,
		and an internal edge from encompassing $B^{(i+1)}_{p+1}$.  
		Combining these results, by \eqref{eq-sixfb} applied to $\cD'$, we have
		\[ 
			6 \leq 2\sum_v (3-d(v)) + (4-i(B_{\cD_i}))+(4-i(B^{(i+1)}_p)) \leq 6.
		\]
		This implies that $B^{(i+1)}_p$ and $B_{\cD_i}$ do not meet along an (internal) edge,
		and that every vertex has degree three, which is impossible.
		
		\vspace{2mm}
		\noindent\textbf{Claim 2:} If $p < q$ and $B^{(i+1)}_p \cap B^{(i+1)}_q \neq \emptyset$, then 
		$B^{(i+1)}_p$ and $B^{(i+1)}_q$ meet at a single vertex in $\partial \cD_i$, or
		they meet along a single edge adjacent to $\cD_i$.
		(In the latter case, necessarily $q=p+1$.)
		
		Suppose not.  Then by Lemma~\ref{lem-onesixth-nice-faces}, we
		find a subdiagram $\cD' \subset \cD/\cD_i$ including
		$B_{\cD_i}$, $B^{(i+1)}_p$ and $B^{(i+1)}_q$ and all they encompass (which includes at least one face).
		By Claim 1 and assumption $A(i)$, every internal face of $\cD'$ has degree at least six.
		Again, \eqref{eq-sixfb} applies to show that
		\[
			6 \leq 2\sum_v (3-d(v)) + (4-i(B_{\cD_i}))+(4-i(B^{(i+1)}_p))+(4-i(B^{(i+1)}_q)),
		\]
		which is at most $9$.
		However, each of the three intersections between $B_{\cD_i}$, $B^{(i+1)}_p$ and $B^{(i+1)}_q$ contributes
		an additional $-2$ to the equation above, either through an internal edge or a vertex of degree four.
		This gives a contradiction.
		
%
		
%
		
		\vspace{2mm}
		\noindent\textbf{Claim 3:} The geodesic edge $\gamma$ of $\cP$
		which meets $\cD_i$ only meets $B^{(i+1)}_1 \cup \cdots \cup B^{(i+1)}_{l_{i+1}}$ along a
		single edge of $B^{(i+1)}_1$ and a single edge of $B^{(i+1)}_{l_{i+1}}$.
		
		By Lemma~\ref{lem-geodesicloop} we know that $\gamma$ meets
		$B^{(i+1)}_1$ and $B^{(i+1)}_{l_{i+1}}$ along a single edge.
		Suppose it also meets $B^{(i+1)}_p$, for some $1 < p < l_{i+1}$.
		It suffices to consider the case when
		$B_{\cD_i}$, $B^{(i+1)}_1$, $B^{(i+1)}_p$ and a geodesic segment of $\gamma$ enclose
		a subdiagram $\cD' \subset \cD/\cD_i$.  Again we remove all vertices of degree two.
		
		As above, every internal face of $\cD'$ has at least six internal edges, 
		and since $\gamma$ is a geodesic, every external face has at least four internal edges,
		with the possible exceptions of $B_{\cD_i}$, $B^{(i+1)}_1$, and $B^{(i+1)}_p$.
		Using similar arguments to the two claims above, we derive a contradiction from \eqref{eq-sixfb}.
		
		\vspace{2mm}
		\noindent\textbf{Claim 4:}
		$A(i+1)$ holds.
		
		The first assertion of $A(i+1)$ follows from Claims 1 and 3 above.
		The second and third assertions follow from Claim 2:
		The faces $B^{(i+1)}_1$ and $B^{(i+1)}_{l_{i+1}}$ have one external and at least
		$d_{Ext}(\lambda) - 2$ internal edges which lie in $\partial \cD_{i+1}$.		
		Likewise, the faces $B^{(i+1)}_2, \ldots, B^{(i+1)}_{l_{i+1}-1}$ 
		have at least $d_{Int}(\lambda)-4$ internal edges in $\partial \cD_{i+1}$.
		
		\vspace{2mm}
		Finally, we can use $A(i)$ to bound $l_{i+1}$.
		Every internal edge of $\partial \cD_i$ contributes a different face to $l_{i+1}$,
		with up to $l_i -1$ exceptions.  Thus
		\begin{align}
			l_{i+1} & \geq  2(d_{Ext}(\lambda) - 2) + 
				(l_i-2)(d_{Int}(\lambda)-4)-(l_i-1) \notag \\ \label{eq-face-rate}
				& \geq (l_i-2)(d_{Int}(\lambda)-5) 
				= (l_i-2) \left\lfloor \tfrac1\lambda - 4 \right\rfloor,
		\end{align}
		completing the proof of the lemma.
	\end{proof}
	
	Since $l_2 \geq 9$, by induction \eqref{eq-face-rate} implies that
	$l_i \geq 9$ for all $i$, and in fact that
	\begin{equation}
		l_i \gtrsim \left\lfloor \tfrac1\lambda - 4 \right\rfloor^i. \label{eq-face-rate2}
	\end{equation}
	We now return to bounding the number of faces in a chain joining $B_0$
	to one of the other sides of $\cP$.
	If there is no $(k+1)$-chain of faces, then every face in $\cD_k \setminus \cD_{k-1}$, with the
	exception of two, is an interior face of $\cD$.
	Therefore by \eqref{eq-face-rate2},
	\[
		F_I \geq l_k \gtrsim \left\lfloor \tfrac1\lambda - 4 \right\rfloor^k.
	\]
%
%
%
	On the other hand, by Lemma~\ref{lem-int-face-bound},
	\[
		F_I \leq \frac{3n-6}{d_{Int}(\lambda)-1} \leq n,
	\]
	thus for some $C = C(\lambda)$, we have
	\[
		k \log \lfloor \tfrac{1}{\lambda}-4 \rfloor \leq \log(n) + C. \qedhere
	\]
\end{proof}

\begin{proof}[Proof of Theorem~\ref{thm-sc-ac-link}]
	Proposition~\ref{prop-facesthin} shows that any geodesic $n$-gon $\cP$ is $\Delta$-slim, with
	\[
		\Delta = \frac{M}{2 \log \lfloor \tfrac{1}{\lambda}-4 \rfloor} \cdot \log(n) + C',
	\]
	where $M/2$ is the maximum diameter of a face, for $M = \max\{|r| \mid r \in R\}$,
	and $C' = 2MC$, with $C$ as above.  So
	the Cayley graph $\Gamma(G,S)$ is $AC_u(\kappa)$ with
	\[
		\kappa = - \frac{4}{M^2} \log^2 \lfloor\tfrac{1}{\lambda}-4 \rfloor.\qedhere
	\]
\end{proof}


\section{Building a round tree in the Cayley complex}\label{sec-round-tree}

Our final goal is to prove the following result.
\begin{theorem}\label{thm-sc-cdim-lower}
	Suppose $G = \langle S | R \rangle$ is a $C'(1/8 -\delta)$ presentation, 
	with $|S| = m \geq 2$ and $|R| \geq 1$, where $\delta \in (0, 1/8)$
	and $|r| \in [3/\delta, M]$ for all $r \in R$.
	Suppose further that for some $M^* \geq 12$, 
	every reduced word $u \in \langle S \rangle$ of length $M^*$
	appears at least once in some cyclic conjugate of some relator $r^{\pm 1}, r \in R$.
	Then for some universal constant $C>0$, we have
	\[
		\Cdim(\bdry G) \geq 1 + C\log(2m) \cdot \frac{M^*}{\log(M)}.
	\]
\end{theorem}
Note that all relators will have size at least $8(M^*-1) \geq 88$.
In fact, if we assume that we have a $C'(1/11)$ presentation ($\delta = 3/88$), then the
assumption $|r| \geq 3/\delta$ is redundant: all relators have size
at least $11(12-1) = 121 \geq 88 = 3/\delta$.

We split the proof of the theorem into two parts.
In this section, we build a round tree in the Cayley
complex $\Gamma^2 = \Gamma^2(G,S,R)$, whose
branching is controlled by the size of $M^*$ relative
to $M$.
In Section~\ref{sec-lower-cdim} we use a lemma of Bourdon
to give the lower bound for the conformal dimension of the boundary.

We recall the following standard definition.
\begin{definition}
	The Cayley complex $\Gamma^2(G,S,R)$ of a finitely presented group
	$G = \langle S | R \rangle$ is the universal cover of the complex $X$,
	where $X$ has a bouquet of $|S|$ oriented circles as a $1$-skeleton, each
	labelled with a generator from $S$, and there are $|R|$ discs glued in with
	boundary labels from the corresponding relators in $R$.
\end{definition}
Note that the $1$-skeleton of $\Gamma^2(G,S,R)$ is the Cayley graph $\Gamma^1(G,S)$.

\subsection{Preliminary lemmas}

We need the following two lemmas of Champetier.
We translate his proofs here for the reader's convenience.

\begin{lemma}[Champetier {\cite[Lemma 4.19]{Cha-95-rand-grps}}]\label{lem-champ-extend-geo}
	Consider a $C'(1/6)$ presentation of a group $G = \langle S | R \rangle$ 
	with Cayley graph $\Gamma(G,S)$, and all relators of length at least seven.
	
	For every point $a \in \Gamma$, there are at most two distinct $s \in S \cup S^{-1}$ that
	satisfy $d(1, as) \leq d(1, a)$.  In other words, any geodesic from $1$ to $a$
	can be extended to any of the neighbors of $a$, with at most two exceptions.
\end{lemma}
Of course, for $a \in G, a \neq 1$, there exists $s \in S \cup S^{-1}$ so that
$d(1,as) = d(1,a)-1$.
We denote a geodesic between $p$ and $q$ by $[p,q]$.
\begin{proof}
	Let $\gamma_1 = [a,1]$, and let $b \in \gamma_1$ satisfy $d(a,b)=1$.
	Suppose there is some $c \in \Gamma$, $c \neq b$, so that $d(a,c)=1$ and
	$d(1,c) \leq d(1,a)$.  Let $\gamma_2 = [c,1]$.
	Note that $a,b$ do not lie in $\gamma_2$.
	
	Let $\cD$ be a reduced diagram for the geodesic triangle $\gamma_1, [a,c], \gamma_2$.
	By Lemma~\ref{lem-bigon-face}, this diagram has
	a face labelled by some $r\in R$ with at most one interior edge, containing $a,b,c$ in
	its boundary.  
	Let $w_i = \gamma_i \cap \partial r$, for $i=1,2$.
	Since $w_i$ are geodesics, and the interior edge has length at most $|r|/6$,
	we have $|w_i| \geq |r|/2-|r|/6 -1 > |r|/6$, for $i=1,2$.
	
	Suppose now that there is another point $c'$ satisfying the same conditions as $c$.
	Then, as before one builds a geodesic triangle from $a,b,c'$, and finds a relator
	$r'$ so that the initial segment of $\gamma_1$ which overlaps $r'$ has length
	at least $|r'|/6$.
	Thus, by the $C'(1/6)$ condition, $r$ and $r'$ are the same relator, and so
	$c=c'$.
\end{proof}

\begin{lemma}[Variation of Champetier {\cite[Lemma 4.20]{Cha-95-rand-grps}}]\label{lem-champ-relator-parts}
	Consider a $C'(1/8)$ 
	presentation of a group $G = \langle S | R \rangle$ with Cayley graph $\Gamma = \Gamma(G,S)$.
	
	For every $u' \in \Gamma$, there is at most one 
	$u \in \Gamma$ so that $d(1,u)=d(1,u')+d(u',u)=d(1,u')+3$,
	and so that a geodesic $\gamma_u = [u,1]$ starts with a 
	subword of some relator $r \in R$ of length greater than $|r|/4+3$.
\end{lemma}
\begin{proof}
	Suppose $u\in\Gamma$ is such a point, and $\gamma_u$ is such a geodesic.
	
	\vspace{2mm}
	{\noindent \textbf{Case 1:}} $u' \notin \gamma_u$.
	
	Then the two geodesics $\gamma_u$ and $\gamma_{u'} = [u,u']\cup [u',1]$ form a geodesic
	bigon that splits at a vertex of $[u,u']$, and so by Lemma~\ref{lem-thin-triangles}
	there is a relator $r' \in R$ whose boundary
	meets both $\gamma_u$ and $\gamma_{u'}$ from the point they split
	in a segment of length at least $|r'|/4$.
	Thus $r$ and $r'$ are the same relators, and so
	the only possibility is that $\gamma_u$ and $\gamma_{u'}$ split at $u$ and each
	begin with $|r|/4+3$ of $r$ and its inverse respectively.
	
	\vspace{2mm}
	{\noindent \textbf{Case 2:}} $u' \in \gamma_u$.
	
	Suppose we have two such points $u,v$ with corresponding geodesics $\gamma_u, \gamma_v$,
	and relators $r, r'$. 
	By Case 1, we can assume that these geo\-des\-ics both pass through $u'$,
	and so $\partial r$ and $\partial r'$ will meet along a subword
	$w$ of $\gamma_u \cap \gamma_v$ that includes $u'$, 
	before $\gamma_u,\gamma_v$ split at some point $p$ around a relator $r''$.
	As $\gamma_u, \gamma_v$ are both geodesics, after $p$ they have
	to include at least $3|r''|/8$ of the relator $r''$.
	
	If $|w| < |r|/8$ and $|w| < |r'|/8$, then $r$ and $r'$ will both meet
	$r''$ along at least one eighth of their length, and so $r,r',r''$ are all the
	same relator, and thus $d(u,1)<d(u',1)$, a contradiction.
	Thus $|w| \geq |r|/8$ or $|w| \geq |r'|/8$, and so $r$ and $r'$ are the same relator,
	and $u=v$ as desired.	
\end{proof}

%

\subsection{Building a round tree}

Suppose $Y$ is a two dimensional complex with a CAT$(\kappa)$ metric, $\kappa<0$,
and there is an $S^1$ action on $Y$ that has a unique fixed point.
If, additionally, there is a tree embedded in $Y$ that meets every $S^1$ orbit in a single
point, then we say $Y$ is a \emph{round tree} \cite[7.C$_3$]{Gro-91-asymp-inv}.

Our goal in this section is to build a 2-complex $A$ which
is topologically embedded in the Cayley complex $\Gamma^2$, 
and whose 1-skeleton is a quasi-convex subset of the Cayley graph $\Gamma^1$.
The complex $A$ will be quasi-isometric to a sector of a round tree; 
we abuse terminology and simply refer to $A$ as a round tree.

The ideas in this section are inspired by the arguments of Champetier~\cite{Cha-95-rand-grps}
 and Bourdon~\cite{Bou-95-cdim-hyp-build}.
However, unlike Champetier, we build more than just a single (or finite number)
of arcs in the boundary.
Unlike Bourdon, we do not have a particular nice hyperbolic building to work in.


We build the round tree inductively.
The round tree at step $n$ is denoted by $A_n$.  
Its branching is controlled by the index set $T = \{1, \ldots, 3\cdot (2m-2)^{K-3}\}$,
where $K = \lfloor M^*/2 - 3 \rfloor$.

Each complex $A_n$ is a union of complexes $A_{\ba_n}$ indexed by $\ba_n \in T^n$,
homeomorphic to a closed disc, which can each be thought of as a triangular region
with left edge a geodesic $L_{\ba_n}$ from $1$, right edge a geodesic $R_{\ba_n}$ from $1$, and 
outer edge a path $E_{\ba_n}$, where $\ba_n \in T^n$.
The left tree is $L_n = \bigcup L_{\ba_n}$, and the right tree is $R_n = \bigcup R_{\ba_n}$,
where the unions are over all $\ba_n$ as above.

\subsubsection{Initial step}
Let $L_\emptyset=[1,s_1], R_\emptyset=[1,s_2]$ be two distinct edges 
from the identity in $\Gamma$.
Combined, $L_\emptyset$ and $R_\emptyset$ give the reduced word $w = s_1^{-1} s_2$ of length $2$.
Choose some relator $r \in R$ which contains $w$ as a subword, and let
$A_0$ be the face corresponding to $r$ in $\Gamma^2$ 
which contains $w$ as a sub-word in its boundary.
Let $E_\emptyset$ be the path of length $|r|-2$ joining $s_1$ to $s_2$ along $\partial A_0$.

\subsubsection{Inductive step}
Assume we have built $A_n$.  Let us fix $\ba_n = (a_1, \ldots, a_n) \in T^n$, and
use the notation $E = E_{\ba_n} \subset \partial A_{\ba_n}$ for the peripheral path 
joining the endpoints of $L_{\ba_n}$ and $R_{\ba_n}$.

Consider the function $d(1,\cdot)$ along $E$.
By induction, this distance is always at least $n$,
and strict local minima are separated by a path of length at least $50$.
(This follows from the fact that every relator has length at least $88$, and
Lemma~\ref{lem-A-good-geodesics} below.)
At points $p \in E$ that are not strict local minima for $d(1,\cdot)$, there
is at least one generator $s \in S$ that leaves $E$ and extends the distance
to the identity by one, i.e., $d(1,ps) = d(1,p)+1$, by Lemma~\ref{lem-champ-extend-geo}.
(In fact there are at least $2m-3$ such extensions.)

We can split the path $E$ into segments of length $6$ centered on local minima,
and of length $3$ or $4$ in-between.  For each endpoint $z$ of the segments we have
an edge that leaves $E$ and extends the distance to the identity by one.
This can be further extended two more steps to give four points at a distance
$d(1,z)+3$ from the identity.
By Lemma~\ref{lem-champ-relator-parts} at least three of these points will not
satisfy the conclusion of the lemma. (This will be useful to us later in the proof.)
We then extend geodesics from each of these three points $K-3$ times using Lemma~\ref{lem-champ-extend-geo},
branching $2m-2$ times at each step.
This gives us $|T|=3\cdot (2m-2)^{K-3}$ distinct points at a distance $d(1,z)+K$ from the identity.

Now for each $a_{n+1} \in T$, we have a corresponding geodesic of length $K$ leaving the endpoints
of each segment in $E$.  Adjacent paths, and the segment between them, concatenate
to give a path of length at most $K+6+K \leq M^*$, so there is some relator having
this word as a subpath.  Add these faces to $A_{\ba_n}$ to define $A_{\ba_{n+1}}$,
where $\ba_{n+1} = (a_1, \ldots, a_n, a_{n+1})$.
We let $L_{\ba_{n+1}}$ be the union of $L_{\ba_n}$ and the path of length $K$ extending from its endpoint
corresponding to $a_{n+1}$.  We define $R_{\ba_{n+1}}$ likewise.
The outer edges of the faces in $A_{\ba_{n+1}} \setminus A_{\ba_{n}}$
(that is, the portion of their boundaries not in $E$,
one of the geodesics of length $K$, or an adjacent face) 
concatenate to give a path $E_{\ba_{n+1}}$ joining
the endpoints of $L_{\ba_{n+1}}$ and $R_{\ba_{n+1}}$ together.
Part of this process is illustrated in Figure~\ref{fig-build-A}.
\begin{figure}
	\begin{center}
	\includegraphics[width=0.9\textwidth]{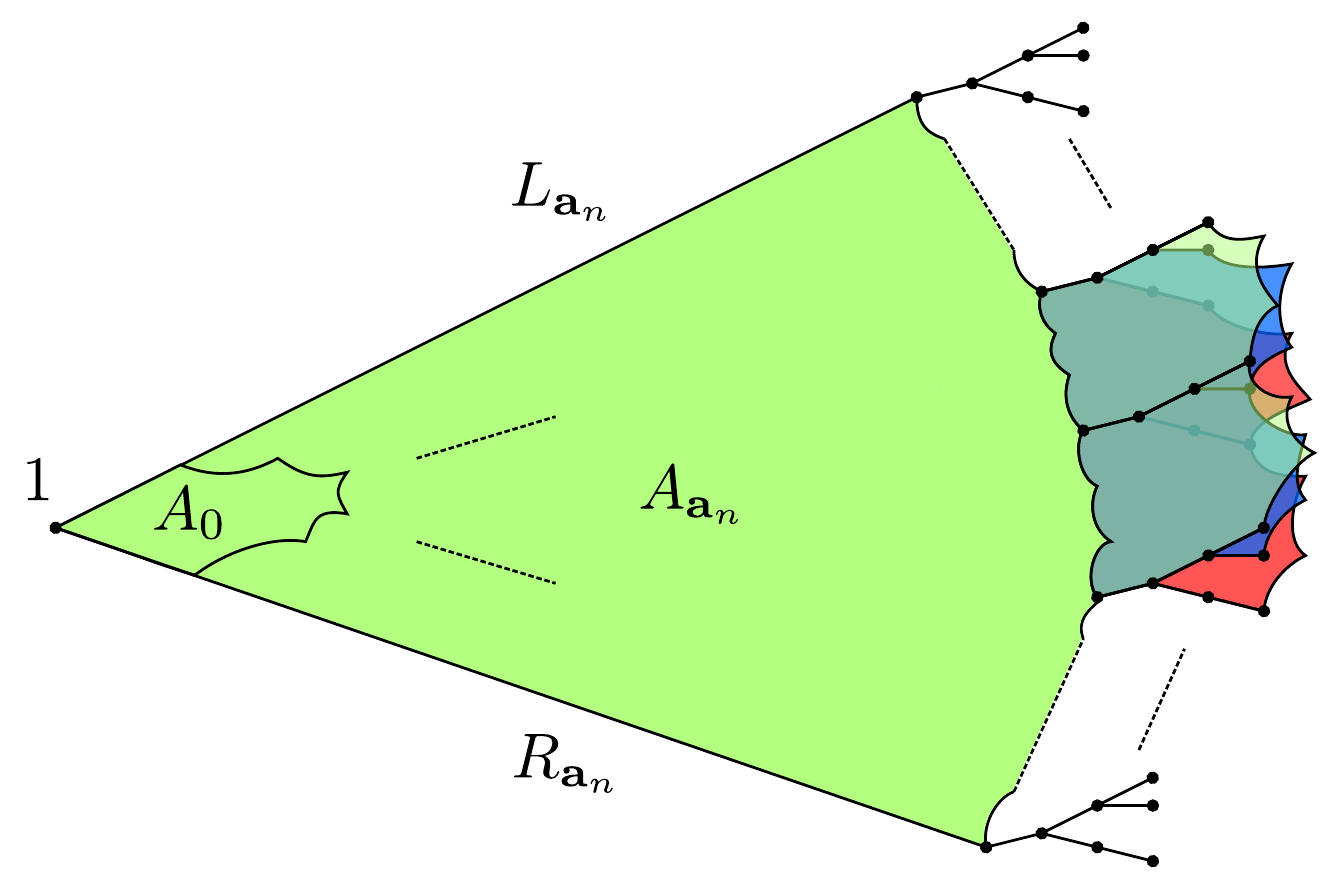}
	\end{center}
	\caption{Building $A_{\ba_{n+1}}$}\label{fig-build-A}
\end{figure}

We show that $E_{\ba_{n+1}}$ does not get closer than $n+1$ to the identity in $\Gamma$.
To be precise:
\begin{lemma}\label{lem-A-good-geodesics}
	Suppose $u',v' \in E_{\ba_{n}}$ are consecutive endpoints of segments,
	$u,v \in E_{\ba_{n+1}}$ are the corresponding points in $E_{\ba_{n+1}}$ (after simplifying
	the path), and $\gamma_{uv} \subset E_{\ba_{n+1}}$ the path joining them, coming from 
	some relator $r \in R$.
	We show that any geodesic from $p\in \gamma_{uv}$ to the identity must pass through
	$u$ or $v$, and include the corresponding sub-path of $\gamma_{uv}$.
\end{lemma}
\begin{proof}
	Suppose some geodesic $\gamma_{1p}$ joins $p$ to the identity without passing through
	$u$ or $v$.  
	We can assume that the edge of $\gamma_{1p}$ adjacent to $p$ is not in $\gamma_{uv}$.
	We can also assume that $d(v',p) \leq d(u', p)$.
	Let $\gamma_{1v}$ be the geodesic path joining $v$ to $1$ (through $v'$).
	
	Consider the closed path formed by $\gamma_{1p},\gamma_{1v},[p,v]_{\gamma_{uv}}$, and
	the associated reduced diagram.
	One can glue on a face to this diagram, labelled with the relator $r$, along
	$[p,v]_{\gamma_{uv}}$ and part of $\gamma_{1v}$.
	Notice that in this diagram, the $r$ face and the face containing $1$ 
	are the only two with exterior edges that are not geodesics.
	Therefore, by Lemma~\ref{lem-bigon-face},
	these faces each have one interior edge, and the diagram has the standard 
	form described in the lemma.
	
	The face adjacent to $r$ is labelled by a relator $r'$, which contains
	$[p,v]_{\gamma_{uv}}$ in its boundary,
	and also the edge of $\gamma_{1p}$ adjacent to $p$, unlike $r$.
	So the relators $r$ and $r'$ must be distinct.
	Thus their overlap is at most $|r'|/8$, and includes 
	$[v',v]_{\gamma_{1v}} \cup [p,v]_{\gamma_{uv}}$.
	Since $\gamma_{1p}$ and $\gamma_{1v}$ are both geodesics,
	$[1,v']_{\gamma_{1v}}$ must contain at least
	$|r'|/2 - 2|r'|/8 = |r'|/4$ of the relator $r'$, which
	contradicts the choice of the paths $[v',v]$.
\end{proof}

\subsubsection{Properties of $A$}
We have built an infinite polygonal complex $A = \bigcup_{n\in\N} A_n$.
It is the union of planar complexes $A_{\ba} \subset A$ indexed by
$\ba = (a_1, a_2, \ldots) \in T^\N$, given by
$A_\ba = \bigcup_{n\in\N} A_{(a_1, \ldots, a_n)}$.

Each $A_{\ba}$ will carry a CAT($-1$) metric, however $A$ may not since
the links of the vertices $v'$ as above have simple closed paths of length
two (created by the $|T|$ different faces all joined along their edges at $v'$).
Before we consider different metrics on $A$, we need to understand how it sits inside $\Gamma^2$.

The complex $A$ was built abstractly, but with an obvious natural 
polygonal immersion $i:A \ra \Gamma^2$.  Denote the $1$-skeleton of $A$ by $A^1$.

\begin{lemma}\label{lem-A-top-embed}
	The map $i:A \ra \Gamma^2$ is a topological embedding.
	
	More precisely, for every $p \in A$, every geodesic joining
	$i(p)$ to $i(1)=1$ in $\Gamma^1$ is the image under $i$ of a geodesic 
	joining $p$ to $1$ in $A^1$.
\end{lemma}
\begin{proof}
	By construction, there is at least one geodesic $\gamma$
	joining $i(p)$ to $1$ in $A^1$.
	Suppose there is some geodesic $\gamma' \subset \Gamma^1$ joining
	$i(p)$ to $1$, whose first edge is not in $i(\St(p))$.
	So $\gamma$ and $\gamma'$ form a bigon, and so by Lemma~\ref{lem-thin-triangles}
	there is some relation $r \in R$ so that the
	first $3|r|/8$ of $\gamma$ after $i(p)$ is a subword of $r$.
	
	The geodesic $\gamma$ is made up of segments in the boundary of relators in $A$,
	and special length three extensions that, by Lemma~\ref{lem-champ-relator-parts},
	do not have any geodesic to the identity which begins with
	a subword of length $|r'|/4+3$ of any relator $r' \in R$.
	
	Thus no such length three subword appears in the first 
	$3|r|/8 - (|r|/4+3) = |r|/8 -3$ vertices of $\gamma$.
	Therefore, $|r|/8 - 3 = |r|(1/8-3/|r|) \geq |r|(1/8 - \delta)$ 
	of $r$ bounds a relator in $A$,
	so $r$ is in $A$, contradicting the hypothesis that $\gamma'$ left $i(\St(p))$.
\end{proof}

\begin{lemma}\label{lem-A-qconvex}
	Consider $A^1$ and $\Gamma^1$ with their path metrics $d_A$ and $d_\Gamma$.
	Then the map $i:A^1 \ra \Gamma^1$ is a quasi-isometric embedding.
	
	In other words, $(A^1,d_A)$ is quasi-isometric to $(A^1,d_\Gamma)$,
	where $d_\Gamma$ is the pullback $d_\Gamma(x,y) = d_\Gamma(i(x),i(y))$.
\end{lemma}
\begin{proof}
	Take any $x,y \in A^1$.
	Since $i$ is a topological embedding, clearly
	$d_A(x,y) \geq d_\Gamma(x,y)$.
	
	Consider the geodesic triangle in $\Gamma^1$ between $1$, $x$ and $y$ with edges
	$\gamma_{1x}$, $\gamma_{1y}$ and $\gamma_{xy}$.
	In light of Lemma~\ref{lem-thin-triangles}, consider the
	structure of a reduced diagram $\cD$ for this triangle.
	
	The geodesics $\gamma_{xy}$ and $\gamma_{1x}$ form a spur
	starting at $x$ that ends at a vertex or an interior edge of $\cD$
	joining $p \in \gamma_{xy}$ to $p' \in \gamma_{1x}$.
	Likewise, $\gamma_{xy}$ and $\gamma_{1y}$ form a spur
	starting at $y$ that ends at a vertex or an interior edge of $\cD$
	joining $q \in \gamma_{xy}$ to $q' \in \gamma_{1y}$.
	Also, $\gamma_{1x}$ and $\gamma_{1y}$ form a spur
	starting at $1$ that ends at a vertex or an interior edge of $\cD$
	joining $p'' \in \gamma_{1x}$ to $q'' \in \gamma_{1y}$.
	
	We claim that $d_A(p'',q'') \leq M/8$.  Either $p''=q''$,
	or $p''$ and $q''$ lie on an interior edge of a relator in the spur starting at $1$.
	If this relator lies in $A$, then we are done.
	Otherwise, by the same argument as in Lemma~\ref{lem-A-top-embed},
	its two external edges in $\gamma_{1x}$ and $\gamma_{1y}$ have length at most 
	$|r|/4+3+|r|(1/8-\delta)$.  We also know that the two internal edges
	have length at most $|r|(1/8 - \delta)$.
	Therefore the boundary of $|r|$ has length at most
	\[
		2(|r|/4+3+|r|(1/8-\delta)) + 2|r|(1/8 - \delta) = |r|(1-4\delta+6/|r|) < |r|,
	\]
	a contradiction.
	
	Since we chose $p,p',p'',q,q',q''$ to make the spurs as long as possible,
	the analysis of Lemma~\ref{lem-thin-triangles} shows that
	$[p,q]_{\gamma_{xy}}$ is adjacent to at most three faces in $\cD$,
	thus $d_\Gamma(p,q) \leq 3M/2$.
	
	Similarly, $d_\Gamma(p',p'')=d_A(p',p'') \leq 3M/2$ and
	$d_\Gamma(q',q'')=d_A(q',q'') \leq 3M/2$.
	Now, recall that $d_\Gamma(p,p'),d_\Gamma(q,q') \leq M/8$, so
	\begin{align*}
		d_A(x,p') & = d_\Gamma(x,p') \leq d_\Gamma(x,p)+M/8,\quad \text{and}\\
		d_A(y,q') & = d_\Gamma(y,q') \leq d_\Gamma(y,q)+M/8.
	\end{align*}
	
	Combining all these results, we see that
	\begin{align*}
		d_A(x,y) & \leq d_A(x,p')+d_A(p',p'')+d_A(p'',q'')+d_A(q'',q')+d_A(q',y)\\
			& \leq \left(d_\Gamma(x,p)+\frac{M}{8}\right) + \frac{3M}{2} + 
				\frac{M}{8} + \frac{3M}{2}
				+ \left( d_\Gamma(q,y) +\frac{M}{8}\right)\\
			& \leq d_\Gamma(x,y) + \frac{27M}{8}. \qedhere
	\end{align*}
\end{proof}


\section{A lower bound for conformal dimension}\label{sec-lower-cdim}

In this section, we will build a model space $X$ quasi-isometric to $A^1$, and
show that $\Cdim(\bdry X)$ has the desired lower bound.
Since we have a quasi-symmetric inclusion of $\bdry A$ into
$\bdry \Gamma = \bdry G$, this will complete the proof of Theorem~\ref{thm-sc-cdim-lower}.

Let $X$ be the graph with a vertex for each face in $A$,
and an edge between two vertices if the boundaries of the corresponding 
faces have non-empty intersection.

\begin{lemma}\label{lem-A-qi-to-X}
	$(A^{1},d_A) \stackrel{\text{q.i.}}{\simeq} (X,d_X)$
\end{lemma}
\begin{proof}
	Let $f:X \ra A^{1}$ be a map that sends each vertex $x \in X$ to
	some vertex in $A$ on the edge of the corresponding face.  
	Clearly, every point in $A^{1}$ is within a $d_A$-distance of $M/2$
	from some point in $f(X)$.
	
	If $d_X(x,y)=1$ for $x,y \in X$, then $d_A(f(x),f(y)) \leq M$, where $M$ is
	the maximum perimeter of a face.  Thus for any $x,y \in X$,
	$d_A(f(x),f(y)) \leq M d_X(x,y)$.
	
	Each edge in a geodesic $[f(x),f(y)] \subset A^1$ is the edge
	of some face in $A$, and adjacent edges will give intersecting faces
	(by definition).
	Adding the faces for $x$ and $y$ to this chain, shows that
	$d_X(x,y) \leq d_A(f(x),f(y)) + 2$.
\end{proof}

We recall the relevant lemma of Pansu and Bourdon.

\begin{lemma}[{\cite[Lemma 1.6]{Bou-95-cdim-hyp-build}}]\label{lem-bourdon-bound}
	Suppose $Z$ is a compact metric space containing a family of
	curves $\cC = \{\gamma_i : i\in I \}$, with diameters uniformly
	bounded away from zero.
	
	Suppose further that there is a probability measure $\mu$ on $\cC$
	and constants $C>0$, $\sig > 0$ such that for all balls $B(z,r)$ in $Z$ 
	\[
		\mu(\{ \gam \in \cC | \gam \cap B(z,r) \neq \emptyset \}) \leq C r^\sig.
	\]
	Then the conformal dimension of $Z$ is at least $1+\frac{\sig}{\tau-\sig}$,
	where $\tau$ is the packing dimension of $Z$, and in fact $\tau-\sig \geq 1$.
\end{lemma}

We need to estimate $\sig$ and $\tau$ for $Z = \bdry X$.

By Lemma~\ref{lem-A-qconvex}, any geodesic in $(A^1,d_A)$ is within a uniformly
bound\-ed Hausdorff distance from a geodesic with the same endpoints in $\Gamma^1$.
Thus $(A^1,d_A)$ is also Gromov hyperbolic.
Since $X$ is quasi-isometric to $A^1$, it too is Gromov hyperbolic.
The boundary $\bdry X$ of $X$
carries a visual metric $\rho$ with parameter $\epsilon$, for some $\epsilon>0$.

In other words, for all points $u,v \in \bdry X$, connected by
a bi-infinite geodesic $\gamma_{uv} \subset X$,
\[
	\rho(u,v) \asymp e^{-\epsilon (u \cdot v)},
\]
where $(u \cdot v) = d(1, \gamma_{uv})$, and $\asymp$ indicates
a multiplicative error of $C_\rho \geq 1$.

Let $\cC = \{ \bdry A_\ba \subset Z : \ba \in T^\N \}$, where 
we abuse notation by identifying $A_\ba$ with $f(A_\ba) \subset X$.
These curves have diameters uniformly bounded away from $0$ since
any geodesic in $X$ asymptotic to the endpoints of $\bdry A_\ba$
must pass uniformly close to the origin.
There is a natural probability measure $\mu$ on $T^\N$ so that, for fixed
$b_1, \ldots, b_n \in T$,  $n \in \N$,
\[
	\mu(\{ (a_1, a_2, \ldots) \in T^\N : a_i = b_i, 1 \leq i \leq n \}) = |T|^{-n}.
\]

\begin{lemma}
	For this choice of $\cC$, $\rho$, $\mu$, we can take
	$\sig = (\log |T|)/\eps$.
\end{lemma}
\begin{proof}
	Fix some $z \in \bdry X$.  Then $z$ lies in the boundary
	of some $A_\ba$, $\ba = (a_1, a_2, \ldots) \in T^\N$.
	
	Suppose $w \in B(z,r) \subset \bdry X$.
	Then $w$ lies in the boundary
	of some $A_\mathbf{b}$, $\mathbf{b} = (b_1, b_2, \ldots) \in T^\N$,
	and
	\[
		\frac{1}{C_\rho} e^{-\eps (z \cdot w)} \leq \rho(z,w) \leq r,
	\]
	so
	\[
		(z \cdot w) \geq (-1/\eps) \log(C_\rho r).
	\]
	Suppose $a_n \neq b_n$, and $n$ is the smallest such $n$.
	Then deleting all vertices in $X$ at distance $n$ from the root face
	will disconnect the boundaries of $A_\ba$ and $A_\mathbf{b}$,
	and so $\gamma_{zw}$ must pass within $n$ of the root face.
	Thus $(z \cdot w) \leq n$, so if $w \in B(z,r)$, $a_m = b_m$ for
	all
	\[
		m < (-1/\eps) \log(C_\rho r).
	\]
	Therefore, 
	\begin{align*}
		\mu(\{\mathbf{b} \in T^\N : A_\mathbf{b} \cap B(z,r) \neq \emptyset \})
		& \leq |T|^{(1/\eps) \log(C_\rho r) + 1} \\
			& \leq |T|^{1+\log(C_\rho)/\eps} \cdot r^{(\log|T|)/\eps}.\qedhere
	\end{align*}
\end{proof}

It remains to bound $\tau$.
When we built $A_{n+1}$ from $A_n$, we added $|T|$ faces to each segment
along $E_{n}$.  So each face in $A_n$ that bordered $E_n$ could have
at most $M|T|$ faces joined on to it.
This gives a way to label every point in $\bdry X$ by an element
of
\[
	W = \{1, 2, \ldots, M|T|\}^\N,
\]
and we denote that labelling by $f:\bdry X \ra W$, which is an injection.

Put the metric $\rho_W$ on $W$, where
\[
	\rho_W((a_1,a_2,\ldots),(b_1,b_2,\ldots))
	= \exp(-\eps \min\{n : a_n \neq b_n\}).
\]
Then $f^{-1} : f(\bdry X) \ra \bdry X$ is a Lipschitz bijection, so
\[
	\tau = \dimP(\bdry X) \leq \dimP(f(\bdry X)) \leq \dimP(W) = \log(M|T|)/\eps.
\]

Thus, by Lemma~\ref{lem-bourdon-bound},
\begin{align*}
	\Cdim(\bdry X) & \geq 1+ \frac{\sig}{\tau-\sig}
		\geq 1+ \frac{\log(|T|)/\eps}{\log(M|T|)/\eps - \log(|T|)/\eps} \\
		& = 1 + \frac{\log |T|}{\log(M)}.
\end{align*}
Since $|T| = 3 \cdot (2m-2)^{K-3}$,
$K = \lfloor M^*/2 - 3\rfloor$,
and $M^* \geq 12$, we have
\[
	\log |T| = \log(3/8)+K \log(2m-2) 
		\geq C M^* \log(2m),
\]
for $C = 1/100$.

Finally,
\[
	\bdry X \stackrel{\text{q.s.}}{\simeq} \bdry A 
		\stackrel{\text{q.s.}}{\subset} \bdry \Gamma = \bdry G,
\]
so we conclude that
\begin{qedequation*}
	\Cdim(\bdry G) \geq
	\Cdim(\bdry X) \geq 1+ C\log(2m) \cdot \frac{M^*}{\log(M)}. \qedhere
\end{qedequation*}


\bibliographystyle{plain}
\bibliography{biblio}

\end{document}